\documentclass[12pt]{article}
\usepackage[margin=1in]{geometry}
\usepackage{amsmath}
\usepackage{mathtools}
\usepackage{amsthm,bm}
\usepackage{bbm}
\usepackage{textcomp}
\usepackage{graphicx}
\usepackage{eurosym}
\usepackage{amssymb}
\usepackage{subcaption}
\usepackage{comment}
\usepackage{cancel}
\usepackage{float}
\usepackage[numbers]{natbib}
\newcommand\citeayn[2][]{\citeauthor*{#2} (\citeyear{#2}) \cite[#1]{#2}}

\usepackage[unicode,breaklinks]{hyperref}
\usepackage[nameinlink]{cleveref}
\usepackage{xcolor}

\definecolor{red}{rgb}{0.7,0.15,0.15}
\definecolor{green}{rgb}{0,0.5,0}
\definecolor{blue}{rgb}{0,0,0.7}
\hypersetup{colorlinks, linkcolor={blue},citecolor={green}, urlcolor={red}}
\definecolor{bleudefrance}{rgb}{0.19, 0.55, 0.91}

\newtheorem{lemma}{Lemma}[section]
\newtheorem{assumption}{Assumption}[section]
\newtheorem{proposition}{Proposition}[section]

\newtheorem{corollary}{Corollary}[section]

\usepackage{tikz}
\usepackage{pgfplots}
\usepackage{enumerate}
\pgfplotsset{compat=1.17}

\def \R{\mathbb{R}}
\def\erm{\mathrm{e}}
\def\drm{\mathrm{d}}
\def\Nu{K}

\title{A Mean Field Game for Capacity Expansion Modeling}
\author{Emma Hubert\thanks{Department of Operations Research \& Financial Engineering, Princeton University.} \thanks{Research partially supported by the NSF grant DMS-2307736.} \and Dimitrios Lolas\footnotemark[1] \and Ronnie Sircar\footnotemark[1]}
\date{\today}

\begin{document}
\maketitle

\begin{abstract}
This paper studies the optimal investment behavior of renewable electricity producers in a competitive market, where both prices and installation costs are influenced by aggregate industry activity. We model the resulting crowding effects using a mean field game framework, capturing the strategic interactions among a continuum of heterogeneous producers. The equilibrium dynamics are characterized via a coupled system of Hamilton--Jacobi--Bellman and Fokker--Planck equations, which describe the value function of a representative producer and the evolution of the distribution of installed capacities over time. We analyze both deterministic and stochastic versions of the model, providing analytical insights in tractable cases and developing numerical methods to approximate the general solution. Simulation results illustrate how aggregate investment responds to changing market conditions, cost structures, and exogenous productivity shocks. 



\end{abstract}

\section{Introduction}

As nations accelerate their transition towards low-carbon economies, the expansion of renewable energy sources, particularly solar and wind, has become a critical priority. However, this shift also brings substantial new challenges for electricity systems. 
A striking illustration occurred in April 2025, when much of Spain and parts of Portugal experienced a massive power outage affecting millions of households (see, for instance \citeayn{bajo-buenestado2025iberian}). Although investigations are ongoing at time of writing, the incident highlights how increasing reliance on weather-dependent generation can amplify vulnerabilities to unexpected disturbances. More broadly, the growing prevalence of negative electricity prices in countries with deep renewables penetration further exemplifies the economic and operational stresses induced by this transition. 

These phenomena underline the urgent need for careful planning of capacity expansion, not only to meet decarbonization targets but also to ensure system resilience. A variety of complementary solutions have been explored, such as demand-side response programs, investments in grid-scale storage, reinforcement of transmission networks, and dynamic pricing mechanisms.
Yet, long-term investment decisions by competing firms remain central to the evolution of electricity markets, shaping both the pace and the stability of renewables integration. Capacity expansion planning must therefore balance the twin goals of promoting rapid deployment and preserving economic and operational robustness. In this context, mathematical modeling—including dynamic games, mean-field approximations, and stochastic control approaches—plays a crucial role in understanding strategic interactions between firms and guiding the design of efficient policy interventions (see for example \citeayn{cacciarelli2025impact}).

When modeling investment decisions in general, and in particular in energy systems, a common approach is to formulate dynamic control or stopping problems, often under uncertainty. In such frameworks, a central planner or a representative agent optimizes a long-term objective, accounting for evolving technological, market, and regulatory conditions (see, for example, \citeayn{aid2009structural}, \citeayn{carmona2010valuation}, \citeayn{ludkovski2016technology}). However, when considering renewable energy installations, such as solar panels which can be deployed relatively easily and at varying scales by a multitude of actors, it becomes essential to explicitly model the decentralized and competitive nature of investment. The price of electricity, which depends on the aggregate level of installed capacity, becomes a key channel of interaction between these agents. Directly analyzing multi-player dynamic games, however, quickly leads to computational intractability as the number of participants grows. To overcome this difficulty, it is common to study the limiting regime where the number of players tends to infinity, leading to more tractable mean field approximations that capture the aggregate effect of individual behaviors.

Mean field games (MFGs) 
study the limiting behaviour of stochastic differential games as the number of (exchangeable) players tends to infinity (see, for instance, \citeayn{gueant2011mean} for an early review article). 
In this framework, each agent optimizes their objective function, taking as given the aggregate behavior of the population, which is itself determined in equilibrium as a fixed point of the collective dynamics of the individual strategies. The MFG formulation has since been widely adopted to analyze decentralized decision-making in large populations, providing both analytical and numerical tractability, and insight into emergent macro-level behavior. 

In recent years, MFGs have been increasingly applied to a wide range of problems in energy systems. These include renewable investment decisions by \citeayn{aid2020entry}, \citeayn{alasseur2023mean}, \citeayn{escribe2024mean} or \citeayn{aid2025regulation}, electricity storage optimization by \citeayn{alasseur2020extended}, and the design of demand-response programs by \citeayn{elie2021meanfield}. MFGs have also been used to study energy mix transitions by \citeayn{chan2017fracking} and carbon emissions regulation by \citeayn{carmona2022mean}, \citeayn{hernandezsantibanez2023pollution}, \citeayn{bichuch2024stackelberg}, \citeayn{dayanikli2024multipopulation}, as well as market design and electricity price formation by \citeayn{firoozi2022principal} or \citeayn{bassiere2024meanfield}. More recently, the framework has been extended to analyze broader aspects of the energy transition, by \citeayn{dumitrescu2024energy}, and energy expenditure from cryptocurrency mining by \citeayn{li2024mean}. This list is far from exhaustive, but illustrates the growing role of MFGs in understanding the complex dynamics of modern energy systems. 

Among the aforementioned papers, we motivate our work from \cite{alasseur2023mean}, which develops a deterministic mean-field-type model for capacity expansion in solar energy with an infinite time horizon. In this model, identical producers optimize solar production capacity, and interact through the electricity price, which depends of the total installed capacity. This framework captures the trade-off between investment rewards, costs and the price-depressing effects of overcapacity. This paper characterizes the steady-state equilibrium using Pontryagin’s maximum principle, and provides a study of the effects of subsidies for solar production from a central planner's point of view.

The model we propose can be seen as a finite-horizon reformulation of \cite{alasseur2023mean}, with two key extensions. First, we allow for heterogeneity in initial capacities across producers. This extension is particularly relevant in the context of solar energy, where small-scale residential investors may compete with large-scale solar farms. Second, we include interaction not only through the state (installed capacity) but also through the control (installation rate), introducing a crowding effect. Indeed, while producers naturally interact through the market price of electricity, which depends on total installed capacity, installation costs may also be influenced by the aggregate speed of investment across the market, for instance, due to supply-chain congestion or labor constraints. 

We employ an approximation where aggregates (or sums) are replaced by a continuum mean multiplied by the number of agents in the sum. This has been used in models of cryptocurrency mining in \cite{li2024mean} and \citeayn{garcia_crypto}, and its error is analyzed in \citeayn{cags}.  It allows finite agent games with interaction through aggregates to benefit from mean field game technology, in particular reduction of dimension. In the current model, 
the aforementioned features lead naturally to a mean field game of both state and control, allowing us to study dynamic investment behavior, threshold effects, and the evolution of aggregate capacity over time. 

Finally, unlike \cite{alasseur2023mean}, we consider a finite time horizon, which better reflects the long but bounded planning horizons of investors in the energy sector. For example, investors may anticipate changes in technology or policy over the next two or three decades, such as the emergence of alternative energy sources or regulatory shifts, and may thus optimize over a finite investment window. Nevertheless, the approach we develop here could be naturally reduced to the infinite-horizon setting.

Our analysis is based on characterizing the mean field game equilibrium via a system of coupled Hamilton--Jacobi--Bellman (HJB) and Fokker--Planck (FP) equations. In an homogeneous setting, where all producers start with the same initial capacity, our model reduces to a deterministic control problem similar to \cite{alasseur2023mean}, but with a finite-time horizon, for which we prove existence and uniqueness of solutions to a forward-backward ODE system and identify threshold behavior in optimal investment. In the heterogeneous case, we study the full mean field game of state and control, establish the structure of the equilibrium using partial differential equation (PDE) methods, and analyze the resulting distribution of installed capacities. In the case of a linear price function, we propose a numerical method based on a quadratic ansatz in the non-installation region and a finite difference scheme in the installation region. We also extend the model to include idiosyncratic randomness in the capacity dynamics, showing that our framework remains tractable under stochastic perturbations.

The remainder of the paper is organized as follows. In \Cref{sec:model}, we introduce our model and its main features. In \Cref{section_hom}, we analyze the homogeneous agents case and derive explicit characterizations of the equilibrium. In \Cref{sec:heterogeneous}, we turn to the heterogeneous setting, study the associated HJB--FP system, and present our numerical scheme. Finally, in \Cref{sec:randomness}, we show how our framework extends to include randomness in capacity dynamics. We conclude in \Cref{sec:conclusion} with a discussion of extensions and policy implications.

\section{Model formulation}\label{sec:model}

The formulation of our model echoes that of \cite{alasseur2023mean}, but is adapted to a finite time horizon \( T > 0 \) (in years) and designed to account for potential heterogeneity in the initial capacities of investors. 

We first consider the finite-player model with \( N+1 \) renewable producers, indexed by \( i \in \{0, \dots, N\} \). Each producer is characterized by their installed capacity \( x^{(i)}_t \) in {MW} at time \( t \in [0, T] \). In the absence of new investment, installed capacity is assumed to depreciate exponentially at rate \( \delta > 0 \), reflecting a gradual loss in production efficiency over time. A producer \( i \in \{0, \dots, N\} \) can invest in new capacity by choosing an installation rate \( \nu^{(i)}_t \geq 0 \) at each time $t \in [0,T]$, which represents the rate in {MW/year} at which new capacity is added. The dynamics of installed capacity for producer \( i \) then follow:
\begin{align}\label{eq:prod_dynamics}
    \drm x^{(i)}_t = \big(-\delta x^{(i)}_t + \nu^{(i)}_t\big)\, \drm t, \quad \text{for } t \in [0, T],
\end{align}
starting from a fixed initial installed capacity \( x^{(i)}_0 \geq 0 \), and where \( \delta > 0 \) is the depreciation rate as above. The total (or aggregate) installed capacity is denoted by
\begin{align}\label{eq:sum_prod}
    X_t \coloneq \sum_{j=0}^{N} x^{(j)}_t, \quad t\in[0,T].
\end{align}
Note that since the installation rate is assumed to be non-negative, we necessarily have
\begin{align}\label{eq:lower_bound}
    X_t \geq \erm^{-\delta t} \sum_{j=0}^{N} x^{(j)}_0, \quad t\in[0,T].
\end{align}

The aggregate installed capacity impacts the price at which all producers sell electricity. More precisely, we assume that the common market price at time $t$ is given by \( P(\lambda_tX_t) \) (in \$/MWh), for a (strictly) decreasing price function \( P: \R_+ \longrightarrow \mathbb{R} \), reflecting the idea that higher total capacity depresses the market price of electricity due to an abundance of cheap renewable production. The parameter $\lambda_t$ represents the fraction of installed capacity that is effectively producing electricity at time $t$, and thus accounts for the intermittency of renewable energy sources, such as variability in wind or sunlight. For simplicity, we assume $\lambda_t\equiv \lambda\in(0,1]$ is constant over time and we absorb this constant to $P$. 

Following \cite{alasseur2023mean}, the marginal cost of production is constant, denoted by $c > 0$ (in \$/MWh), and represents operational expenses associated with running the installed capacity. This may include maintenance, insurance and other ongoing management costs of operating a renewable energy asset such as a solar or wind installation. We also model the cost of installing new capacity through two distinct mechanisms that curb rapid expansion:
\begin{enumerate}[$(i)$]\itemsep -2pt
    \item $\alpha > 0$ (in \$/MW) is the marginal installation cost;
    \item $\beta > 0$ (in (\$ year)/MW$^2$) is a crowding sensitivity parameter, capturing the additional cost per MW of capacity that arises from increased installation activity in the market. This reflects market frictions such as limited contractor availability, supply chain bottlenecks, or permitting delays that become more pronounced when many producers are simultaneously expanding.
\end{enumerate}
However, contrary to \cite{alasseur2023mean}, we assume that the installation cost is not constant, but instead increases with the total installation rate due to crowding effects. More precisely, the instantaneous cost of installing new capacity $\nu_t^{(i)}$ is defined as follows:
\begin{align}\label{eq:sum_rate}
  \nu_t^{(i)} \big( \alpha  + \beta \Nu_t \big), \quad \text{with } \; \Nu_t \coloneq \sum_{j =0}^N \nu_t^{(j)}, \quad t \in [0,T].
\end{align}

Finally, given a fixed initial installed capacity \( x^{(i)}_0 \geq 0 \), the individual producer's objective is defined by the following maximization problem
\begin{equation}\label{eq:cost_general}
J_i \big(x^{(i)}_0 \big) := \sup_{\nu^{(i)}\geq0}\,\int_0^{T}  \erm^{-rt} \Big(  \big( P(X_t) - c \big) h x_{t}^{(i)} - \alpha \nu_t^{(i)} - \beta 
\nu^{(i)}_t \Nu_t
\Big) \drm t,    
\end{equation}
where $h \in [0,\,8760]$ is the number of production hours per year, and $r>0$ is a discount rate.

As is well known, this type of multi-agent non-zero sum differential game is difficult to solve, even for the two-player case. For analytical and computational tractability, we will instead study a continuum  approximation by a type of mean field game, where a continuum mean multiplied by the number of players proxies for the aggregate quantities $X$ and $\Nu$, respectively defined in \eqref{eq:sum_prod}, \eqref{eq:sum_rate}. We first consider in the next section the simpler case of homogeneous producers; the study of the general mean field game is postponed to \Cref{sec:heterogeneous}.

\section{Homogeneous producers}
\label{section_hom}

In this section, we assume that all producers start with the same initial installed capacity, $x_0^{(i)} \eqcolon x_0 \geq 0$ for all $i \in \{0, \dots, N\}$. This simplification, which makes our model a finite-horizon counterpart of \cite{alasseur2023mean} and thus facilitates comparison, will be relaxed in the next section.

In this setting, producers are indistinguishable, and their optimal actions can be derived using Pontryagin's maximum principle, as in \cite{alasseur2023mean}.

More precisely, we first focus on the $i$th producer's optimization problem, $i \in \{0,\dots,N\}$, for fixed capacity and effort of other producers. The Hamiltonian corresponding to the objective function \eqref{eq:cost_general} is given by
\begin{align*}
    H^{(i)}_t \big(x_{t}^{(i)}, u_{t}^{(i)}, \nu_{t}^{(i)} \big) \coloneq u_{t}^{(i)} \big( - \delta x_{t}^{(i)} + \nu_{t}^{(i)} \big) + \big( P(X_t) - c \big) h x_{t}^{(i)} - \alpha \nu_t^{(i)} - \beta 
\nu^{(i)}_t \Nu_t, \quad t \in [0,T],
\end{align*}
where ${u}^{(i)}$ is the adjoint variable, or costate. Recalling that $K$ defined in \eqref{eq:sum_rate} depends on $\nu^{(i)}$, the maximization of the previous Hamiltonian leads to the following installation rate
\begin{align*}
\nu_t^{(i)}= \dfrac{1}{2\beta} \bigg(u_{t}^{(i)}-\alpha-\beta\sum_{j\neq i}\nu_{t}^{(j)}\bigg)^+, \quad t \in [0,T].
\end{align*}
Under the assumption of homogeneous producers, all adjoint variables ${u}^{(i)}$, $i \in \{0, \dots, N\}$, actually coincide, as well as the corresponding optimal investment rates, which can thus be derived by solving the previous equation, leading to
\begin{equation*}
\nu_{t}^{(i)}= \dfrac{1}{\beta(N+2)} \big(u_t^{(i)}-\alpha \big)^+, \quad t\in [0,T].
\end{equation*}
Furthermore, the costate ${u}^{(i)}$ satisfies the following ordinary differential equation (ODE):
\begin{align*}
    \dot{u}_{t}^{(i)} = (r + \delta) u_{t}^{(i)} - (P(X_t)-c)\,h-h  x_{t}^{(i)}\frac{\partial P(X_t)}{\partial x_{t}^{(i)}}, \quad t \in [0,T], \quad u_{T}^{(i)}=0.
\end{align*}
We will subsequently drop the dependence on the index $i$, since all producers are identical.

As the number of producers is assumed to be large, we follow here the standard mean field game convention by assuming that the impact of an individual producer to the total capacity is negligible, \textit{i.e.}
$$\frac{\partial P(X_t)}{\partial x_{t}^{(i)}} \approx 0.$$
Replacing in the above ODE for the costate, we derive the following simplified ODE,
\begin{align*}
\dot{u}_{t}=(r+\delta)u_{t}-(P(X_t)-c)h,\quad t \in [0,T], \quad u_{T}^{}=0.
\end{align*}
Still under the assumption that $N$ is large, the natural approximation $\frac{N+1}{N+2}\approx1$ leads to the following formula for the total installation rate defined in \eqref{eq:sum_rate},
\begin{align*}
\Nu_t \coloneq \sum_{j =0}^N \nu_t^{(j)} = \dfrac{1}{\beta} \big(u_{t}^{}-\alpha\big)^+,
\end{align*}
highlighting in particular that installation of new capacity occurs at time $t \in [0,T]$ if and only if $u_t> \alpha$. 
Finally, summing \eqref{eq:prod_dynamics} over $i \in \{0, \dots, N\}$ and using the previous equation, we derive the following ODE for the total installed capacity defined in \eqref{eq:sum_prod},
\begin{align*}
\dot{X}_t+\delta X_t=\frac{1}{\beta}(u_t-\alpha)^+, \quad t \in [0,T], \quad X_0=(N+1)x_0.
\end{align*}

In summary, when $N$ is sufficiently large, the original multi-player non-zero-sum game with homogeneous producers can be asymptotically approximated by a mean field game, characterized by the following forward-backward system:
\begin{subequations}\label{eq:FB_system}
\begin{align}
\dot{X}_t &= - \delta X_t + \frac{1}{\beta}(u_t-\alpha)^+, \quad \quad \quad \; \, t \in [0,T], \quad X_0=(N+1)x,
\label{eq:FB_X}\\
\dot{u}_{t} &=(r+\delta)u_{t}-\left(P(X_t)-c\right)h, \quad t \in [0,T], \quad u_{T}^{}=0.\label{eq:FB_u}
\end{align}
\end{subequations} 
The previous system is the finite-horizon analogue of the one derived in \cite{alasseur2023mean}. In the next section, we first prove existence and uniqueness for the system \eqref{eq:FB_X}--\eqref{eq:FB_u} under some regularity assumptions for the price function $P$. We then highlight some interesting properties of the solution and the associated optimal installation rate. These theoretical results offer preliminary insights into the investment problem, which are further developed and illustrated by the numerical simulations presented at the end of this section.

\subsection{Theoretical results}

As mentioned above, the first main result establishes existence and uniqueness of the forward-backward system \eqref{eq:FB_X}--\eqref{eq:FB_u}.

\begin{proposition}\label{prop:existence_uniqueness}
    Assuming that the price function $P$ is Lipschitz continuous, then there exists a unique solution to the system \eqref{eq:FB_X}--\eqref{eq:FB_u}. 
\end{proposition}

To prove the previous result, we first study the forward system analogous to \eqref{eq:FB_system}. More precisely, we consider the following system of ODEs,
\begin{align}\label{eq:F_system}
\dot{X}_t &= - \delta X_t + \frac{1}{\beta}(u_t-\alpha)^+, \quad \dot{u}_{t} =(r+\delta)u_{t}-\left(P(X_t)-c\right)h, \quad t \in [0,T],
\end{align}
with initial condition $(X_0,u_0) \in \R_+ \times \R$.

\begin{lemma}\label{lem:ODE_forward}
    If $P$ is Lipschitz continuous, then there exists a unique solution $(X,u)$ to the initial value problem \eqref{eq:F_system}. Moreover, a comparison result holds, in the sense that if $(X,u)$ and $(\tilde X, \tilde{u})$ are solutions to \eqref{eq:F_system}, starting from $(X_0,u_0)$ and $(X_0, \tilde{u}_0)$ respectively with $u_0 < \tilde{u}_0$, then $X_t \leq \tilde X_t$ and $u_t < \tilde u_t$ for all $t \in [0,T]$.
\end{lemma}

\begin{proof}[Proof of \Cref{lem:ODE_forward}]
Letting $\bm{x}_t:=(X_t,u_t)$, we can rewrite the system \eqref{eq:F_system} of two ODEs as $\dot{\bm{x}}_t=f(\bm{x}_t)$, starting from $\bm{x}_0:=(X_0,u_0)$.
Since the price function $P$ is assumed to be Lipschitz continuous, the generator $f$ is also Lipschitz, so existence and uniqueness holds on $[0,T]$ by the Picard–Lindelöf theorem for this initial value problem. 
Note that, in particular, the solution satisfies for $t \in [0,T]$,
\begin{subequations}\label{eq:sol_explicit}
\begin{align}
X_t &=X_0 \erm^{-\delta t}+\frac{1}{\beta}\int_{0}^{t} \erm^{-\delta(t-s)} (u_s-\alpha)^+ \drm s, \label{eq:X_sol}\\ 
u_t &=u_0 \erm^{(r+\delta)t}-h\int_{0}^{t} \erm^{(r+\delta)(t-s)} \big( P(X_s)-c\big)  \drm s. \label{eq:u_sol}
\end{align}
\end{subequations} 

Let then $(X,u)$ and $(\tilde X, \tilde{u})$ be solutions to \eqref{eq:F_system}, starting from $(X_0,u_0)$ and $(X_0, \tilde{u}_0)$ respectively, with $u_0<\tilde{u}_0$. To establish the comparison result, we proceed by contradiction: let $t\in[0,T]$ be the smallest time for which $u_t=\tilde{u}_t$, meaning that $u_s \leq \tilde{u}_s$ for $s \in [0,t]$. From \eqref{eq:X_sol}, we have $X_s \leq \tilde X_s$, and using further that $P$ is decreasing, we deduce $P(X_s) \geq P(\tilde X_s)$, for all $s \in [0,t]$. Using this in \eqref{eq:u_sol}, we obtain
\begin{align*}
u_t 
&\leq u_0 \erm^{(r+\delta)t}-h\int_{0}^{t} \erm^{(r+\delta)(t-s)} \big( P(\tilde X_s)-c \big) \drm s\\
&<\tilde{u}_0 \erm^{(r+\delta)s}-h\int_{0}^{t} \erm^{(r+\delta)(t-s)} \big(P(\tilde X_s)-c \big) \drm s
=\tilde{u}_t,
\end{align*}
which leads to a contradiction, since we assumed $u_t=\tilde{u}_t$. In other words, for all $t \in [0,T]$ we have $u_t<\tilde{u}_t$, implying by \eqref{eq:X_sol} that $X_t \leq \tilde X_t$.
\end{proof}

By the previous lemma, it suffices to show existence and uniqueness of $u_0 \in \R$ such that the solution $(X,u)$ to \eqref{eq:F_system} starting from $(X_0,u_0)$ satisfies $u_T = 0$ to prove \Cref{prop:existence_uniqueness}.

\begin{proof}[Proof of \Cref{prop:existence_uniqueness}]
Let $(X,u)$ and $(\tilde X, \tilde{u})$ be solutions to \eqref{eq:F_system}, starting from $(X_0,u_0)$ and $(X_0, \tilde{u}_0)$ respectively, and assume without loss of generality that $u_0>\tilde{u}_0$. Define then $\Delta(t) \coloneq u_t - \tilde{u}_t$ which is positive by \Cref{lem:ODE_forward} for all $t\in[0,T]$. Recall that we also have $X_t \geq \tilde X_t$, implying $P(X_t) \leq  P( \tilde X_t)$, $t\in[0,T]$. Starting from \eqref{eq:u_sol}, and using in addition that $P$ is $L$-Lipschitz continuous for $L \geq 0$, we obtain,
\begin{align*}
\Delta(T)&= \big( u_0 - \tilde u_0 \big) \erm^{(r+\delta)T} 
+ h\int_{0}^{T} \erm^{(r+\delta)(T-t)} \big(P( \tilde X_t) - P(X_t) \big)  \drm t \\
&\leq \Delta(0) \erm^{(r+\delta)T} 
+ h L \int_{0}^{T} \erm^{(r+\delta)(T-t)} \big(X_t - \tilde X_t \big)  \drm t.
\end{align*} 
Using now \eqref{eq:X_sol}, we have
\begin{align*}
    X_t - \tilde X_t &= \frac{1}{\beta}\int_{0}^{t} \erm^{-\delta(t-s)} \big( (u_s-\alpha)^+ - (\tilde u_s-\alpha)^+ \big) \drm s,
\end{align*}
and substituting into the previous inequality, we deduce
\begin{align*}
\Delta(T)
&\leq \Delta(0) \erm^{(r+\delta)T} 
+ \frac{hL}{\beta} \int_{0}^{T} \int_{0}^{t} \erm^{(r+\delta)(T-t)} \erm^{-\delta(t-s)} \big( (u_s-\alpha)^+ - (\tilde u_s-\alpha)^+ \big) \drm s \drm t \\
&\leq \Delta(0) \erm^{(r+\delta)T}+\frac{hL}{\beta} \int_{0}^{T}\int_{0}^{t} \erm^{(r+\delta)(T-t)} \erm^{-\delta(t-s)} \big(u_s-\tilde{u}_s\big)  \drm s \, \drm t \\
&\leq \Delta(0) \erm^{(r+\delta)T}+\frac{hL M}{\beta} \int_{0}^{T} \Delta(t)\, \drm t,
\end{align*} 
where $M$ is a constant depending on $\delta, r$ and $T$.
Then, by Gronwall's Lemma, we deduce
\begin{align*}
    \Delta(T) \coloneq u_T-\tilde{u}_T \leq (u_0-\tilde{u}_{0}) \erm^{\big(r+\delta+\frac{h L M}{2\beta} \big)T}.
\end{align*}
As the case $u_0<\tilde{u}_0$ can be treated similarly, we conclude that the map $u_0 \in \R \longmapsto u_T \in \R$ is Lipchitz continuous. 

In addition, we can show that this continuous map is negative for $u_0$ sufficiently small, and positive for $u_0$ sufficiently large. Indeed, on the one hand, if
\begin{align*}
    u_0 < h \min_{t\in[0,T]} \int_{0}^{t} \erm^{-(r+\delta)s} \big(P \big(X_0\erm^{-\delta s} \big)-c \big)  \drm s, 
\end{align*}
which is non-positive, so in particular smaller than $\alpha$, then the forward system \eqref{eq:F_system} admits the following solution 
\begin{equation*}
X_t=X_0 \erm^{-\delta t}, \quad u_t=u_0 \erm^{(r+\delta)t}-h\int_{0}^{t} \erm^{-(r+\delta)(s-t)} \big( P(X_s)-c \big) \drm s, \quad t \in [0,T].
\end{equation*}
By the condition on $u_0$ above, we have $u_t <0$ for all $t \in [0,T]$, so in particular $u_T<0$ as wanted. On the other hand, recall that $X_t \geq X_0 \erm^{-\delta t}$ for all $t \in [0,T]$ by \eqref{eq:lower_bound}, implying since $P$ is decreasing that $P(X_t)\leq P(X_0 \erm^{-\delta t})$, which is a positive constant. We then have
\begin{equation*}
u_T\geq u_0 \erm^{(r+\delta)T}-h\int_{0}^{T} \big( P\big(X_0 \erm^{-\delta t}\big)-c \big) \erm^{(r+\delta)(T-t)} \drm t,
\end{equation*}
and thus taking
\begin{align*}
    u_0 > h \int_{0}^{T} \big( P\big(X_0 \erm^{-\delta t}\big)-c \big) \erm^{-(r+\delta)t} \drm t,
\end{align*}
we obtain $u_T>0$. 

From the previous reasoning, we deduce that the continuous map $u_0 \in \R \longmapsto u_T \in \R$ has values of opposite sign on $\R$, and thus has a root by Bolzano's theorem. In other words, there exists $u_0 \in \R$ such that $u_T=0$, which concludes the existence part. Uniqueness follows directly from \Cref{lem:ODE_forward}. Indeed, by the comparison principle, one necessarily needs $u_0 = \tilde u_0$ to ensure $u_T = \tilde u_T = 0$.
\end{proof}

To ensure existence and uniqueness of the solution to the forward-backward system \eqref{eq:FB_X}--\eqref{eq:FB_u}, we now work under the assumption that the strictly decreasing price function $P$ is Lipschitz continuous.
In addition to this main result, we can establish some properties of the total installed capacity over time. We begin by constructing a trivial solution under a specific condition on the model parameters.

\begin{lemma}\label{lemma_trivial}
If the parameters $(r,\delta,T,h,\alpha,c,X_0)$ and the price function $P$ are such that
\begin{align*}
    h\int_{t}^{T} \erm^{-(r+\delta)(s-t)}(P(X_0 \erm^{-\delta s})-c)\, \drm s\leq\alpha, \quad \text{ for all } \; t \in [0,T],
\end{align*}
then the unique solution to the system \eqref{eq:FB_X}--\eqref{eq:FB_u} is given by
\begin{equation}\label{eq:trivial_sol}
(X_t,u_t)=\left(X_0 \erm^{-\delta t},h\int_{t}^{T} \erm^{-(r+\delta)(s-t)}(P(X_0 \erm^{-\delta s})-c)\, \drm s\right), \quad t \in [0,T],
\end{equation}
with corresponding optimal installation rate $\Nu \equiv 0$. 
\end{lemma}
\begin{proof}
Consider the candidate solution \eqref{eq:trivial_sol} provided in the lemma. By assumption, one has $u_t \leq \alpha$ for all $t \in [0,T]$, implying that 
\begin{equation*}
\Nu_t = \frac{1}{\beta} \big( u_t-\alpha\big)^+=0, \quad t \in [0,T].
\end{equation*}
It is then straightforward to verify that the candidate solution \eqref{eq:trivial_sol} satisfies \eqref{eq:FB_X}--\eqref{eq:FB_u}.
\end{proof}
The previous result highlights that if $T$ or $h$ are sufficiently small, or conversely if $r$, $c$, $\alpha$ or $X_0$ are sufficiently large, then there will be no new capacity installed at any time. As such model will be less interesting to study, we are led to make the following assumption, which will be verified in the numerical experimentation in \Cref{numerics}, in particular for the parameters chosen in \eqref{eq:param_numerics}.

\begin{assumption}\label{model_assump}
The model parameters $(r,\delta,T,h,\alpha,c,X_0)$ and the price function $P$ are such that $(P(X_0)-c)h>(r+\delta)\alpha$ and there exists a time $t_0\in[0,T]$ for which 
\begin{equation*}
h\int_{t_0}^{T} \erm^{-(r+\delta)(s-t_0)}\big( P\big(X_0 \erm^{-\delta s}\big)-c\big) \drm s>\alpha.
\end{equation*}
\end{assumption}

Under the previous assumption, the unique solution $(X,u)$ to the system \eqref{eq:FB_X}–\eqref{eq:FB_u} is not trivial, and exhibits some interesting properties, as stated in the following lemmas. 

\begin{lemma}\label{T^* lemma}
Under \Cref{model_assump}, there exists ${T^\star } \in (0,T)$ such that $u_t\geq \alpha$ for $t\leq {T^\star }
$ and $u_t\leq \alpha$ for $t\geq {T^\star }$.
\end{lemma}
\begin{proof}
First, if $u_t \leq \alpha$ for all $t \in [0,T]$, then the corresponding optimal installation rate is $K\equiv 0$, leading to the solution derived in \eqref{eq:trivial_sol}. In particular, we have
\begin{align*}
    u_t = h\int_{t}^{T} \erm^{-(r+\delta)(s-t)}(P(X_0 \erm^{-\delta s})-c)\, \drm s \leq \alpha, \quad t \in [0,T],
\end{align*}
which contradicts \Cref{model_assump}. Thus, there exists at least one $t \in[0,T]$ such that $u_t > \alpha$. 

We then prove by contradiction that $u_0 > \alpha$. In particular, we assume that $u_0\leq \alpha$ and take the first $t \in [0,T]$ for which $u_t=\alpha$, which exists by the above reasoning. In this case, we have $\dot{u}_t \geq0$, and $u_s \leq \alpha$ for $s \in [0,t]$. This gives $K_s = 0$ for all $s \in [0,t]$, implying $X_t=X_0 \erm^{-\delta t} > X_0$.
Using in addition that the price function $P$ is strictly decreasing, we derive from \eqref{eq:FB_u} the following inequality
\begin{equation*}
0\leq \dot{u}_t = (r+\delta)u_t -\left(P(X_t)-c\right)h \leq (r+\delta) \alpha-\left(P(X_0)-c\right).
\end{equation*}
However, by \Cref{model_assump}, the right hand side is negative, leading to a contradiction. 

To summarize, we necessarily have $u_0 > \alpha$ and $u_T = 0$. Then by continuity, there exists a time $T^\star \in [0,T]$ such that $u_{T^\star} = \alpha$. It remains to show that $u_t \geq \alpha$ on $t \in [0,T^\star]$ and $u_t \leq \alpha$ on $t \in [T^\star,T]$. We also proceed by contradiction, assuming that there exists an interval $[t_1, t_2] \subset [0,T]$ such that $u_{t_1}=u_{t_2}=\alpha$ and $u_s < \alpha$ for $s \in (t_1,t_2)$. On this interval, the installation rate is zero and the capacity decays at rate $\delta$. Mathematically, $X_{t_1}>X_{t_2}$, and thus $P(X_{t_1})\leq P(X_{t_2})$. Using again \eqref{eq:FB_u}, we obtain
\begin{equation*}
\dot{u}_{t_1}=(r+\delta)u_{t_1}-(P(X_{t_1})-c)h \geq (r+\delta)u_{t_2}-(P(X_{t_2})-c) h=\dot{u}_{t_2},
\end{equation*}
and since $\dot{u}_{t_1}\leq 0 \leq \dot{u}_{t_2}$, we necessarily have $\dot{u}_{t_1}= \dot{u}_{t_2}=0$, which contradicts the fact that $u_s < \alpha$ for $s \in (t_1,t_2)$. Therefore, such interval cannot exist. In other words, if $u$ drops strictly below $\alpha$ at some point, it cannot cross back to $\alpha$ afterwards. Since $u_0 > \alpha$ and $u_T = 0$, this concludes the proof of existence of ${T^\star }$ as required in the lemma.
\end{proof}

We can now prove that the price will stay above the production cost. 
\begin{lemma}\label{lem:price_cost}
Under \Cref{model_assump}, we have $P(X_t)\geq c$ for all $t \in [0,T]$.
\end{lemma}
\begin{proof}
Let $T^\star \in [0,T]$ given by \Cref{T^* lemma}.
We first prove that $P(X_t)\geq c$ for all $t \in [T^\star,T]$.
At time $T^\star$, we have by definition $u_{T^\star}=\alpha$ and $\dot{u}_{T^\star } \leq 0$, implying using \eqref{eq:FB_u} that
\begin{equation*}
0\geq \dot{u}_{T^\star }
=(r+\delta)u_{T^\star }-(P(X_{T^\star })-c)h
=(r+\delta)\alpha-(P(X_{T^\star })-c)h.
\end{equation*}
In particular, $P(X_{T^\star })-c\geq0$. Moreover, for $t\geq{T^\star }$, we have $u_t \leq \alpha$, and therefore $X$ satisfies the ODE $\dot{X}_t+\delta X_t=0$. In other words, $X$ decay at rate $\delta$ starting from $T^\star$, and thus $P(X_t) \geq P(X_{T^\star }) \geq c$ for $t\geq{T^\star }$.

We now prove the inequality for $t \in [0,T^\star]$, assuming by contradiction that $P(X_s)<c$ for some $s\in[0,T^\star]$. Note that by \Cref{model_assump} and the previous reasoning, we have 
\begin{align*}
    P(X_0)>\dfrac{r+\delta}{h}\alpha + c \geq c, \quad \text{and} \quad P(X_{T^\star }) \geq c.
\end{align*}
Therefore, there must exist an interval $(t_1,t_2) \subset [0,T^\star]$ for $t_1 < t_2 < T^\star$, on which $P(X_t)<c$ and $P(X_{t_1})=P(X_{t_2})=c$, with $P$ being decreasing near $t_1$ and increasing near $t_2$. More precisely, there exists $\varepsilon > 0$ such that $P(X_t)\geq c$ for $t\in [t_1-\varepsilon,t_1] \cup [t_2,t_2+\epsilon]$, while $P(X_t)< c$ for $t\in(t_1,t_2)$. Since $P$ is  strictly
decreasing,
we have $X_{t_1}=X_{t_2}=P^{-1}(c)$, where $P^{-1}$ denotes the inverse of $P$, and $\dot{X}_{t_1}\geq \dot{X}_{t_2}$.
Moreover, since $t_1 < t_2 < T^\star$, we have by \Cref{T^* lemma} that $u_t\geq\alpha$ for $t\in[t_1,t_2]$. The system of equations satisfied by the solution $(X,u)$ on the interval $(t_1,t_2)$ is thus 
\begin{equation*}
\dot{X}_t+ \delta X_t = \frac{1}{\beta}(u_t-\alpha),\qquad
\dot{u}_t=(r+\delta)u_t-h(P(X_{t})-c).
\end{equation*}
Since $P(X_t)< c$ for $t\in(t_1,t_2)$, we have $\dot{u}_t > 0$, and thus $u_{t_1}<u_{t_2}$. However, this is in contradiction with the fact that $X_{t_1}=X_{t_2}=P^{-1}(c)$ and $\dot{X}_{t_1}\geq0\geq \dot{X}_{t_2}$, since
\begin{align*}
    \frac{1}{\beta}(u_{t_1}-\alpha) = \dot{X}_{t_1}+\delta X_{t_1} = \dot{X}_{t_1}+\delta P^{-1}(c) \geq \dot{X}_{t_2}+\delta P^{-1}(c) = \dot{X}_{t_2}+\delta X_{t_2}=\frac{1}{\beta} (u_{t_2}-\alpha).
\end{align*}
This concludes the proof that $P(X_t)\geq c$ for all $t\in[0,T]$.\end{proof}

Under \Cref{model_assump}, and thanks to \Cref{T^* lemma}, we have $u_t\geq\alpha$ for $t\leq {T^\star}$ and $u_t\leq\alpha$ for $t\geq {T^\star}$. In particular, on $t\leq {T^\star }$, the ODE \eqref{eq:FB_X} for $X$ becomes
\begin{align*}
\dot{X}_t+\delta X_t=\frac{1}{\beta}(u_t-\alpha). 
\end{align*}
Differentiating the previous ODE and replacing $\dot{u}$ using \eqref{eq:FB_u} yields the following two-point boundary value problem,
\begin{equation}
\ddot{X}_t=r\dot{X}_t+(r+\delta)\left(\delta X_t+\frac{\alpha}{\beta}\right)-\frac{h}{\beta}\left(P(X_t)-c\right),\label{Xdoubledot}
\end{equation}
with boundary conditions $X_0 \geq 0$ fixed and $\delta X_{T^\star }+\dot{X}_{T^\star }=0$, since $u_{T^\star} = \alpha$.
Then, for $t\geq {T^\star}$, we have $u_t\leq\alpha$, implying in particular that there is no installation after time $T^\star$. We deduce that for $t \in [T^\star,T]$,
\begin{align*}
    X_t=X_{T^\star }  \erm^{-\delta(t-{T^\star })}, \quad u_{T^\star }= h \int_{T^\star }^{T} \erm^{-(r+\delta)(t-{T^\star })} \big(P \big(X_{T^\star } \erm^{-\delta(t-{T^\star })} \big)-c\big) \drm t = \alpha.
\end{align*}

\subsection{Linear price function}\label{limhomogeneous}
In this section, we focus on the price $P$ being a linear decreasing function, \textit{i.e.}
\begin{align*}
    P(x)=d_1-d_2x \quad\mbox{for some } d_1,d_2>0,
\end{align*}
and assume \Cref{model_assump} is satisfied. 
Under this linear specification for the price, we can obtain a semi-explicit formula for the optimal capacity.
\begin{proposition}\label{Proposition}
    For a linear price function as above, we have
\begin{align}\label{eq:linear_sol}
X_t &= 
\begin{cases}
    C \erm^{r_{1} t}+D \erm^{r_{2} t}+\theta, &\mbox{ for }t\leq {T^\star },\\
    X_{T^\star }  \erm^{-\delta(t-{T^\star })}, &\mbox{ for }t\in[{T^\star },T],
\end{cases} \\[1em]
\text{ with } \; \; \theta &\coloneq \frac{{h}(d_1-c)-(r+\delta){\alpha }{}}{\beta(r+\delta)\delta+{h}d_2}, \; 
C \coloneq \frac{(X_0+\theta)(r_2+\delta) -\delta\theta  \erm^{-r_2 T^\star }}
{(r_2+\delta) -  \erm^{(r_1-r_2) T^\star }(r_1+\delta)}, \;
D \coloneq X_0-C-\theta, \nonumber \\
\text{ and } \; r_{1,2} &\coloneq \frac12\bigg(r\pm \sqrt{r^2+4\bigg((r+\delta)\delta+\frac{h}{\beta} d_2\bigg)}\bigg), \nonumber
\end{align}
and where $T^\star$ is solution to the following transcendental equation,
\begin{equation}\label{Tstareqn}
\alpha = \frac{h(d_1-c)}{r+\delta} \big( 1- \erm^{-(r+\delta)(T-{T^\star }
)} \big) -\frac{hd_2}{r+2\delta} \big( 1- \erm^{-(r+2\delta)(T-{T^\star }
)} \big) \big(C \erm^{r_1{T^\star }
}+D \erm^{r_2{T^\star }
}+\theta \big).
\end{equation}
\end{proposition}
\begin{proof}
Recall that $T^\star$ is defined in \Cref{T^* lemma}. As already mentioned, there is no installation after time $T^\star$, thus $X_t=X_{T^\star } \erm^{-\delta(t-{T^\star })}$ for $t \in [T^\star,T]$ as stated in \eqref{eq:linear_sol}. For $t\leq {T^\star }$, using the linear price in \eqref{Xdoubledot} we derive
\begin{equation*}
\ddot{X}_t-r\dot{X}_t-\left((r+\delta)\delta+\frac{h}{\beta} d_2\right)X_t=
\frac{1}{\beta}\left(\alpha(r+\delta)- h(d_1-c)\right),
\end{equation*}
with boundary conditions $X_0 \geq 0$ fixed and $\delta X_{T^\star}+\dot{X}_{T^\star }=0$.
Straightforward, albeit lengthy, computations show that the candidate solution $X$ defined in \eqref{eq:linear_sol} is indeed solution to this second-order linear ODE boundary-value problem. In particular, $r_1,r_2$ are the two solutions of the quadratic equation $x^2-rx-\left((r+\delta)\delta+\frac{h}{\beta}d_2\right)=0$, with $r_1>r+\delta>-\delta>r_2$. The constants $C$ and $D$ are derived from the boundary conditions. We further verify that the denominator in the definition of $C$ is not zero, because $r_2+\delta<0<r_1+\delta$. Finally, we find ${T^\star }\in(0,T)$ from $u_{T^\star } = \alpha$, which leads to \eqref{Tstareqn}. Finding ${T^\star}$ from this equation fully determines $C$, and thus $D$, which describe the solution $X$.
\end{proof}

Since the linear price is Lipschitz continuous, the previous solution is unique by \Cref{prop:existence_uniqueness}, implying in particular that there exists a unique solution $T^\star$ to Equation \eqref{Tstareqn}. This will also be observed numerically in the following section.

\subsection{Numerical solution}\label{numerics}

In this section, we numerically solve \eqref{eq:FB_X}--\eqref{eq:FB_u} for two different price functions $P$, linear---as in the previous section---and inverse, and for the parameters used in \cite{aid2020entry,alasseur2023mean}:
\begin{equation}
\begin{split}\label{eq:param_numerics}
X_0&=30\text{GW},\quad c=15\$/\text{MWh}, \quad r=0.1\text{year}^{-1},\quad
\delta=\frac{\log2}{10}\text{year}^{-1},\\
h&=3000\frac{\text{hours}}{\text{year}},\quad\alpha=1400\$/\text{kW},\quad\frac{1}{\beta}=5\text{MW}^2/(\$\text{year}). 
\end{split}
\end{equation}

The comparison principle in Lemma \ref{lem:ODE_forward} suggests a shooting method to numerically solve the forward-backward system \eqref{eq:FB_X}--\eqref{eq:FB_u}: starting from an arbitrary value $u_0$, one can solve the forward system \eqref{eq:F_system} to obtain $u_T$, and then update the choice of $u_0$ depending on the sign of $u_T$. More precisely, by the comparison principle, one should increase the guess for $u_0$ if $u_T<0$, and decrease otherwise. 
Since both the linear and inverse price functions are assumed to be Lipschitz continuous, this shooting method is guaranteed to converge to the unique solution.

\vspace{-1em}

\paragraph{Linear price function.} The previous numerical scheme is first implemented for a linear price function of the form $P(x)=d_1-d_2x$ with $d_1=500$\$/MWh, $d_2=10^{-2}$\$/MW$^2$h, and compared with another numerical method, consisting of computing $T^\star$ by numerically solving \eqref{Tstareqn} and then replacing in the explicit formulas established in \Cref{Proposition}. The results are displayed in \Cref{fig:two-panel} for a fixed time horizon $T=5$. More precisely, \Cref{numerics_linear_left} illustrates how to find $T^\star$ such that \eqref{Tstareqn} is satisfied. Then, \Cref{numerics_linear_right} confirms that the solution obtained with the shooting method coincides on $[0,T]$ with the semi-explicit solution.

The main economic observation from the numerical results illustrated in \Cref{numerics_linear_right} is the decay of capacity at the natural rate~$\delta$ after the critical time~$T^\star \approx 0.25$. More precisely, the total installed capacity initially increases rapidly, reflecting the strong early incentives to invest. However, due to the finite planning horizon, a critical threshold time~$T^\star$ emerges, beyond which further investment is no longer profitable. This behavior contrasts sharply with the infinite-horizon results of~\cite{alasseur2023mean}, where investment incentives remain constant over time. Indeed, in our finite time horizon setting, since no profits can be generated after time~$T$, producers find it optimal to cease installing new capacity sufficiently in advance of the terminal date to avoid incurring costs that cannot be recouped. In other words, once the remaining time falls below a certain threshold, the opportunity cost of investment outweighs any remaining benefits. Consequently, investment stops, and capacity decays passively at the rate~$\delta$, in line with the theoretical predictions established in \Cref{T^* lemma}. This numerical example thus provides a clear illustration of how finite-horizon considerations fundamentally alter investment dynamics compared to the infinite-horizon benchmark.

\begin{figure}[h!]
    \centering
    \begin{minipage}[t]{0.35\textwidth}
        \centering
        \includegraphics[width=\textwidth]{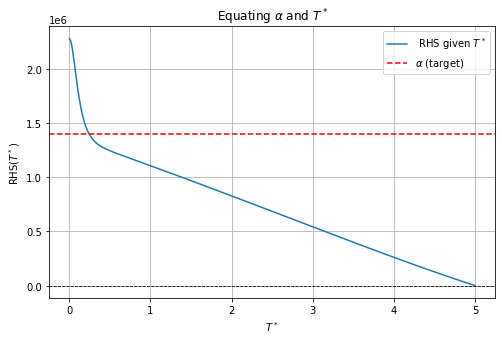}
        \subcaption{\footnotesize $\alpha$ and $T^\star $ equation}
        \label{numerics_linear_left}
    \end{minipage}
    \begin{minipage}[t]{0.35\textwidth}
        \centering
        \includegraphics[width=0.95\textwidth]{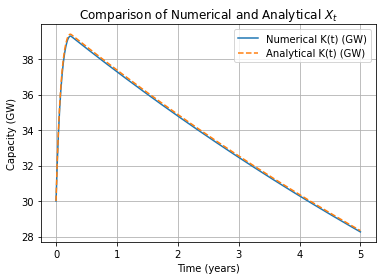}
        \subcaption{\footnotesize Capacity dynamics}
        \label{numerics_linear_right}
    \end{minipage}
    \caption{Capacity dynamics and corresponding $T^\star$ with linear price function}
    \label{fig:two-panel}
\end{figure}

\vspace{-1em}

\paragraph{Inverse price function.} We then consider as in \cite{alasseur2023mean} an `inverse' price function of the form $P(x)=p/x$, for $p=6.5\times10^6\$/\text{h}$ and the parameters previously described in \eqref{eq:param_numerics}. We first emphasize that, although the price function $P(x) = p/x$ is not Lipschitz continuous on $\mathbb{R}_+$, this does not cause issues for the dynamics. Indeed, the installed capacity $X_t$ satisfies the lower bound $X_t \geq X_0  \erm^{-\delta t}$, where $X_0 > 0$ is the initial capacity. Thus, $X_t$ remains bounded away from $0$ for all $t \in [0,T]$, and $P(x)$ is effectively Lipschitz on the reachable set $\{ x \geq X_0  \erm^{-\delta T} \}$. This guarantees the well-posedness of the solution.

Figure~\ref{fig:M2_real_numbers} presents the dynamics of the total capacity~$X$ over time in the case of an inverse price function, for three different finite time horizons ($T=5$, $10$, and $20$ years). We first observe that the solution satisfies \Cref{T^* lemma}: as in the previous case with a linear price, there is an initial increase in total capacity, followed by decay at the natural rate~$\delta$. 

\begin{figure}[htb]
    \centering
    
    \begin{subfigure}[b]{0.3\textwidth}
        \centering
        \includegraphics[width=\textwidth]{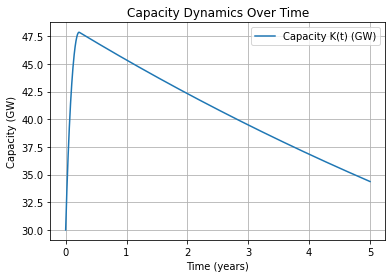}
        \caption{\footnotesize $T = 5$ years}
        \label{fig:M2_real_T5}
    \end{subfigure}
    \hfill
    \begin{subfigure}[b]{0.3\textwidth}
        \centering
        \includegraphics[width=\textwidth]{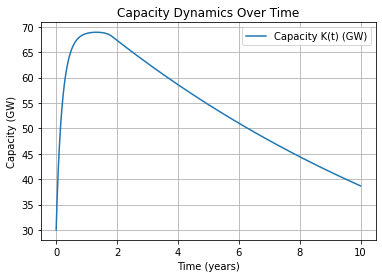}
        \caption{\footnotesize $T = 10$ years}
        \label{fig:M2_real_T10}
    \end{subfigure}
    \hfill
    \begin{subfigure}[b]{0.3\textwidth}
        \centering
        \includegraphics[width=\textwidth]{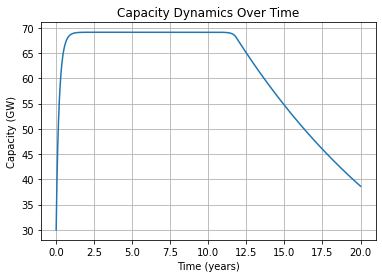}
        \caption{\footnotesize $T = 20$ years}
        \label{fig:M2_real_T20}
    \end{subfigure}
    \caption{Capacity dynamics with inverse price function for different time horizons}
    \label{fig:M2_real_numbers}
\end{figure}

For the longest time horizons ($T=20$ years), the total capacity rapidly approaches the infinite-horizon stationary state studied in~\cite{alasseur2023mean}, and remains near this steady state for most of the planning horizon. This illustrates a form of \emph{turnpike effect} (see, for \textit{e.g.} \cite{carmona2024leveraging}): for sufficiently large maturity, the optimal policy drives the system close to the long-run equilibrium early on and maintains it there as long as `possible', \textit{i.e.} here until it is no longer profitable to invest. Interestingly, when the horizon is larger than $10$ years, the simulations indicate that the cessation of investment occurs approximately $8.5$ years before the terminal date, suggesting that the critical time $T^\star \approx T - 8.5$ for the given parameters. 

In contrast, when the planning horizon is short (\textit{e.g.}, $T=5$), the total capacity increases briefly but remains below the infinite-horizon steady state, before decaying at the natural rate~$\delta$, as also seen in \Cref{numerics_linear_right} for the linear price case with the same maturity.  This behavior reflects the fact that producers anticipate the approaching horizon and adjust their decisions to avoid incurring installation costs that cannot be recovered.

\section{Heterogeneous producers}\label{sec:heterogeneous}
We now approximate the finite player game by one in which there is a continuum of producers whose capacities at time $t=0$ are distributed according to the continuous density function $m_0(x)$ on $x\geq0$, which smooths out the actual discrete (empirical) distributions of initial installed capacities $(x^{(i)}_0)_{i \in \{0,\dots,N\}}$. 

In the approximating MFG, we consider $x_t$ the installed capacity of a representative producer at time $t \in [0,T]$. 
Analogous to \eqref{eq:prod_dynamics}, we have 
\begin{equation}\label{eq:dynamic_x_heterogeneous}
    \drm x^{}_t=(-\delta x^{}_t+\nu^{}_t)\,  \drm t,
\end{equation}
where $\nu_t$ is their installation rate at time $t \in [0,T]$. 
We then let $\bar{x}_t$ and $\bar{\nu}_t$ denote, respectively, the mean installed capacity and mean installation rate at time $t$. 
We adopt the following approximations of the aggregate quantities $X$ and $\Nu$,
\begin{equation}\label{eq:approx_aggregate}
X_t \coloneq \sum_{j=0}^{N}x^{(j)}_t\approx x_t+N\bar{x}_t,\qquad
\Nu_t \coloneq \sum_{j=0}^{N}{\nu}^{(j)}_t\approx {\nu}_t+N{\bar{\nu}_t}, \quad t \in [0,T],
\end{equation}
which restore the individual impacts on the aggregates, in the sense that now
$$\frac{\partial X_t}{\partial x_t} = 1, \quad \text{ and } \quad \frac{\partial \Nu_t}{\partial \nu_t} = 1.$$
This continuum aggregate game approximation was used in the cryptocurrency-mining model in \cite{li2024mean}, and is further analyzed in \cite{cags}. This can also be related to what is called \emph{$\lambda$-interpolated mean field games}, introduced in \cite{carmona2023nash} and generalized in \cite{dayanikli2025cooperation}. As we will see in \Cref{V_concave}, under the preceding approximations and the additional concavity assumption on the function $x \longmapsto xP(x+N\bar{x})$, producers with smaller installed capacity achieve a higher profit per unit of new energy installed than producers with larger capacity. This, in turn, results in a higher installation rate for smaller producers.

Using the approximations \eqref{eq:approx_aggregate} in the original objective \eqref{eq:cost_general}, we are led to consider the following dynamic value function, for $(t,x) \in [0,T] \times \R_+$,
\begin{equation}\label{value}
V(t,x)=\sup_{\nu\geq0}\,\int_{t}^{T} \erm^{-r(s-t)}\Big[(P(x_s + N\bar{x}_s)-c)hx_s-\nu_s(\alpha+\beta(\nu_s+N\bar{\nu}_s))\Big]\, \drm s.
\end{equation}
For fixed continuum mean quantities $\bar{x}$ and $\bar{\nu}$, the representative producer's Hamiltonian is
\begin{align*}
    \sup_{\nu_t \geq0} \left\{\left(  P(x+N\bar{x}_t) - c \right) hx  -\nu_t(\alpha+\beta(\nu_t+N\bar{\nu}_t)) + \frac{\partial V}{\partial x} \left( -\delta x + \nu_t \right) \right\},
\end{align*}
from which we derive the associated optimal response
\begin{equation}\label{nustar}
\nu^\star(t,x)=\frac{1}{2\beta} \bigg(\dfrac{\partial V}{\partial x}-\alpha-\beta N\bar{\nu}_t \bigg)^+, \quad (t,x) \in [0,T] \times \R_+.
\end{equation}
The value function defined in \eqref{value} then satisfies the following Hamilton-Jacobi-Bellman (HJB) equation:
\begin{equation}\label{HJB}
\frac{\partial V}{\partial t}-rV-\delta x\frac{\partial V}{\partial x}+(P(x+N\bar{x}_t)-c)hx+\frac{1}{4\beta} \bigg[ \bigg(\dfrac{\partial V }{\partial x}-\alpha-\beta N\bar{\nu}_t \bigg)^+ \bigg]^2=0,
\end{equation}
for $(t,x) \in [0,T) \times \R_+$, with terminal condition $V(T,x)=0$.
The associated Fokker-Planck (FP) equation for the density $m(t,x)$ of installed capacities at time $t>0$ is 
\begin{equation}\label{FP}
\frac{\partial }{\partial t}m+\frac{\partial }{\partial x}\left[m\left(-\delta x+\frac{1}{2\beta} \bigg(\frac{\partial V}{\partial x}-\alpha-\beta N\bar{\nu}_t \bigg)^+ \right)\right]=0,
\end{equation}
with given initial condition $m(0,x)=m_0(x)$. The preceding system of equations \eqref{HJB}--\eqref{FP} characterizes a mean field game of state and control, approximating the original $N$-player game. Integrating~\eqref{nustar} over~$m(t,\cdot)$ and rearranging yields to the following fixed-point condition for the mean installation rate at equilibrium,
\begin{align}\label{nubarformula}
\bar{\nu}_t=\frac{\int_{\mathcal{A}_t} (\frac{\partial V}{\partial x}-\alpha) m(t,x) \drm x}{2\beta+\beta N\int_{\mathcal{A}_t}m(t,x)\,\drm x},
\quad \text{with } \mathcal{A}_t \coloneq \bigg\{x:\frac{\partial V}{\partial x}-\alpha-\beta N \bar{\nu}_t>0\bigg\}, \quad t \in [0,T].
\end{align}
The dynamics of the mean state at the mean field equilibrium is then given by
\begin{align*}
   \dot{\bar{x}}_t+\delta \bar{x}_t=\bar{\nu}_t, \quad t \in [0,T],
\end{align*}
with initial condition $\bar{x}_0=\int xm_0(x)\,\drm x$.

In \Cref{homogeneous} we will explain how this mean field approximation generalize the case of homogeneous producers studied in \Cref{section_hom}. Then, in \Cref{properties}, we will prove some preliminary results on the value function that will be helpful to solve in \Cref{linear_section} the HJB--FP system \eqref{HJB}--\eqref{FP}. 

\subsection{Connection to the homogeneous producers case}\label{homogeneous}

Under the following simplifications, one recovers the specific setting studied in \Cref{section_hom}:
\begin{enumerate}\itemsep -2pt
    \item[$(i)$] All producers start with the same capacity.
    \item[$(ii)$] There is no direct individual impact on the price.
\end{enumerate}
First, by $(i)$, the producers remain indistinguishable at all times, in the sense that their individual capacities are all identical, and in particular \(m(t,\cdot)\) is a Dirac delta function for every \(t \in [0,T]\). Moreover, their optimal controls coincide, and we thus deduce from \eqref{nustar}
$$\bar{\nu}_t = \nu^\star(t,x_t) = \frac{1}{\beta(N+2)} \bigg(\dfrac{\partial V}{\partial x}(t,x_t)-\alpha \bigg)^+, \quad t \in [0,T],$$
which coincide with \eqref{nubarformula} because either the measure of $m$ is positive in $\mathcal{A}_t$ or it is zero and $\frac{\partial V}{\partial x}\leq\alpha$ on $\mathcal{A}_t$. Using this in the dynamic of the mean state $\bar{x}$ at the mean field equilibrium, with the approximation $\frac{N+1}{N+2}\approx 1$, we derive the required dynamics for the total capacity $X_t = (N+1)\bar{x}_t$, namely
\begin{equation}\label{new_2}
\dot{X}_t+\delta X_t=\frac{1}{\beta}\left(u_t-\alpha\right)^+, \quad \text{with } \; u_t \coloneq \frac{\partial V}{\partial x}(t,x_t), \quad t \in [0,T].
\end{equation}

It remains to compute the dynamics of $u_t$ defined above, the marginal value along the optimal trajectory, in the spirit of Hotelling's rule for the economics of exhaustible resources (see \cite{chan2017fracking} for a MFG version). First, by $(ii)$, an individual producer has no direct impact on the price, meaning that its impact is only through the mean capacity. This boils down to replacing $P(x_t + N\bar{x}_t)$ by $P((N+1)\bar{x}_t)$, or $P(X_t)$ for fixed total capacity $X_t$, $t \in [0,T]$. Using this in \eqref{HJB}, and differentiating with respect to the individual capacity $x$, we obtain   
\begin{align*}
\frac{\partial^2 V}{\partial t \partial x}-r\frac{\partial V}{\partial x}-\delta \frac{\partial V}{\partial x}-\delta x_t\frac{\partial^2 V}{\partial x^2}+(P(X_t)-c)h+\frac{\left(\frac{\partial V}{\partial x}-\alpha-\beta N\bar{\nu_t} \right)^+}{2\beta}\frac{\partial ^2V}{\partial x^2}=0.
\end{align*}
Noticing that $\frac{1}{2\beta} \left(\frac{\partial V}{\partial x}-\alpha-\beta N\bar{\nu_t} \right)^+=\nu^\star(t,x)$ and $-\delta x_t+\nu^\star(t,x_t)=\dot{x}_t$, the above PDE becomes 
\begin{align*}
\frac{\partial^2 V}{\partial t \partial x}-(r+\delta)\frac{\partial V}{\partial x}+h(P(X_t)-c)+\dot{x}_t\frac{\partial ^2 V}{\partial x^2}=0.
\end{align*}
Finally, by definition of $u$ in \eqref{new_2} and using the chain rule $\dot{u}_t=\frac{\partial^2 V}{\partial t\partial x}+\dot{x}_t\frac{\partial^2V}{\partial x^2} $, we deduce 
\begin{equation}\label{new_1}
\dot{u}_t=(r+\delta)u_t-h(P(X_t)-c),
\end{equation}
with $u_T=0$ since $V(T,x)=0$. Therefore, under the two simplications $(i)$--$(ii)$, the general HJB--FP system \eqref{HJB}--\eqref{FP} boils down to the ODE system \eqref{new_1}--\eqref{new_2}, which coincides with the system \eqref{eq:FB_X}--\eqref{eq:FB_u} studied in \Cref{section_hom}. The two methods, Pontryagin maximum principle and HJB, are thus consistent in this setting.

\subsection{Properties of the value function}\label{properties}
In this section, we establish several properties of the value function, which will guide the numerical resolution of \eqref{HJB}--\eqref{FP}. The following result highlights that producers with smaller capacity install at a larger rate. In fact, we will prove that the optimal strategy for producers with large capacity is to stop installing new capacity.

\begin{lemma}\label{V_concave}
Assuming that the revenue $x \longmapsto x P(x+N\bar{x})$ is concave, then for all $t \in [0,T]$, $x \longmapsto V(t,x)$ is also concave, and $x \longmapsto \nu^\star(t,x)$ is non-increasing.
\end{lemma}
\begin{proof}
Let $t \in [0,T]$, and consider two processes $x$ and $x^\prime$ satisfying \eqref{eq:dynamic_x_heterogeneous}, respectively starting from $x_t$, $x'_t$ at time $t$, and driven by two possibly different controls $\nu$, $\nu'$. In particular
\begin{align*}
    x_s=x_t \erm^{-\delta(s-t)}+\int_{t}^{s} \erm^{-\delta(s-y)}\nu_y \drm y, \;
    \text{ and } \;
    x^\prime_s=x^\prime_t \erm^{-\delta(s-t)}+\int_{t}^{s} \erm^{-\delta(s-y)}\nu^\prime_y \drm y, \quad s \in [t,T].
\end{align*}
Let $\lambda \in [0,1]$, define the process $x^\lambda$, starting from $x_t^\lambda \coloneq \lambda x_t + (1-\lambda)x_t^\prime$ at time $t$, and driven by the admissible control $\nu^\lambda \coloneq \lambda\nu+(1-\lambda)\nu'$. Note that we have
\begin{align*}
    x_s^\lambda = \lambda x_s + (1-\lambda)x_s^\prime = x^\lambda_t \erm^{-\delta(s-t)}+\int_{t}^{s} \erm^{-\delta(s-y)}\nu^\lambda_y \drm y, \quad s \in [t,T].
\end{align*}
Let now $s \in [t,T]$. From the concavity assumption on $x \longmapsto (P(x+N\bar{x})-c)x$, we deduce
\begin{align}\label{ineq1}
    &\ \lambda \big( P(x_s+N\bar{x}_s)-c \big) x_s + (1-\lambda) \big( P(x^\prime_s+N\bar{x}_s)-c \big) x^\prime_s
    \leq \big( P \big( x_s^\lambda +N\bar{x}_s \big) - c \big) x_s^\lambda.
\end{align}
Moreover, since the cost is linear-quadratic, we also have
\begin{align}\label{ineq2}
    -\lambda\nu_s(\alpha+\beta \nu_s+\beta N \bar{\nu}_s)-(1-\lambda)\nu'_s(\alpha+\beta \nu'_s+\beta N \bar{\nu}_s)
    \leq - \nu^\lambda_s \big( \alpha+\beta \nu^\lambda_s + \beta N \bar{\nu}_s \big).
\end{align}
Denoting by $f$ the running reward in the definition of the value function $V$ in \eqref{value}, namely
\begin{align}\label{running_reward}
    f(x_s, \nu_s) \coloneq \big( h P(x_s + N\bar{x}_s)-c \big) x_s - \nu_s \big( \alpha+\beta(\nu_s+N\bar{\nu}_s) \big),
\end{align}
we obtain, by summing \eqref{ineq1} and \eqref{ineq2}, the following inequality, 
\begin{align*}
    \lambda f(x_s, \nu_s) + (1-\lambda) f(x^\prime_s, \nu^\prime_s)
    \leq f(x^\lambda_s, \nu^\lambda_s),
\end{align*}
for all $s \in [t,T]$. Multiplying by $\erm^{-r(s-t)}$ and integrating from $t$ to $T$, we deduce
\begin{align*}
    \lambda \int_t^T \erm^{-r(s-t)}  f(x_s, \nu_s) \drm s
    + (1-\lambda) \int_t^T \erm^{-r(s-t)}  f(x^\prime_s, \nu^\prime_s) \drm s
    \leq \int_t^T \erm^{-r(s-t)} f(x^\lambda_s, \nu^\lambda_s) \drm s.
\end{align*}
The RHS of the previous inequality is upper bounded by $V(t,x_t^\lambda) = V(t,\lambda x_t+(1-\lambda)x'_t)$. Finally, taking the supremum over the arbitrary controls $\nu$ and $\nu^\prime$, we deduce
\begin{align*}
    \lambda V(t,x_t)+(1-\lambda)V(t,x'_t) \leq V(t,\lambda x_t+(1-\lambda)x'_t), \quad t \in [0,T].
\end{align*}
Since \(x_t\) and \(x'_t\) are arbitrary, this completes the proof that the value function \(V\) is concave.

The fact that $\nu^\star(t,x)$ is decreasing in $x$ is a direct consequence of the concavity of $V$. Indeed, since $V$ is concave, $\frac{\partial V}{ \partial x}$ is decreasing in $x$, implying that $\nu^\star(t,x)=\frac{1}{2\beta}\big( \frac{\partial V}{\partial x}-\alpha-\beta N\bar{\nu}_t \big)^+$ is non-increasing in $x$.
\end{proof}

In the following, we will work under the standing assumption stated in the previous lemma, namely that the mapping \(x \longmapsto x\,P(x + N\bar{x})\) is concave. This assumption is 
satisfied in 
the examples of price functions we will consider, in particular the linear and inverse cases, discussed in \Cref{linear_section,inverse_price} respectively.

\begin{lemma}\label{max}
Suppose that $\lim_{y\rightarrow\infty}P(y)< c$. Then, there exists $x_{\rm max}<\infty$ such that, for all $t \in [0,T]$ and for any process $x$ with dynamics defined by \eqref{eq:dynamic_x_heterogeneous} and starting from $x_t \geq x_{\rm max} \erm^{-\delta t}$, the associated optimal control is $\nu^\star(s,x_s)=0$ for all $s\in[t,T]$.
\end{lemma}
\begin{proof}
Since $\lim_{y\rightarrow\infty}P(y)< c$, there exists $q < \infty$ such that $P(q)<c$, and we can thus define $x_{\rm max} \coloneq q \erm^{\delta T}$. Let then $t \in[0,T]$. For any process $x$ starting from $x_t \geq x_{\rm{max}} \erm^{-\delta t}$ at time $t$, the non-negativity of both the capacities and the admissible control imply
\begin{align*}
    x_s + N\bar{x}_s \geq x_s \geq x_t \erm^{-\delta (s-t)} \geq x_{\rm{max}} \erm^{-\delta s} \geq x_{\rm max} \erm^{- \delta T} = q, \quad s \in [t,T],
\end{align*}
and since $P$ is decreasing with $P(q)<c$, we deduce that $P(x_s + N\bar{x}_s) - c < 0$ for all $s \in [t,T]$. Since the revenue at any time $s \in [t,T]$, namely $h ( P(x_s + N\bar{x}_s) - c ) x_s$, is negative, and the cost of installing new capacity is always non-negative, it is straightforward to conclude that no capacity should be installed on $[t,T]$. More precisely, let $x$ as above, driven by a non trivial control $\nu$, and starting from $x_t \geq x_{\rm max} \erm^{-\delta t}$, and let $x^\circ$ the corresponding uncontrolled process, \textit{i.e.} $x^\circ_s = x_t \erm^{-\delta (s-t)}$, $s \in[t,T]$. We clearly have $x_s \geq x^\circ_s = x_t \erm^{-\delta (s-t)} \geq q$, and by the previous remark on the negative revenue and non-negative cost, we deduce that 
\begin{align*}
    f (x^\circ_s, 0 ) > f (x_s, \nu_s), \quad \text{ for } \nu_s > 0, \quad s \in [t,T].
\end{align*}
recalling that $f$ is the running reward defined in \eqref{running_reward}. Multiplying by \(\exp\bigl(-r(s - t)\bigr)\) and integrating over \(s \in [t,T]\), it follows that any nonzero control yields a strictly smaller total reward than applying no control. Consequently, the optimal control is \(\nu \equiv 0\). 
\end{proof}
By the previous lemma, we have that at any time $t$, $\nu^{\star}{(t,x)}=0$ for $x$ sufficiently large, namely $x \geq x_{\rm max} \erm^{-\delta t}$. From Equation \eqref{nustar}, we deduce that $\frac{\partial V}{\partial x}\leq\alpha+\beta N \bar{\nu}_t$ for $x\geq x_{\rm max}$.
Since $V$ is concave, the mapping $x \longmapsto \frac{\partial V}{\partial x}$ is decreasing. As a result, we define for $t \in [0,T]$ and fixed mean installation rate $\bar{\nu}_t$,
\[
x^\star(t) =
\begin{cases}
\min \left\{ x \geq0 \;\middle|\; \frac{\partial V}{\partial x}(t, x) = \alpha + \beta N \bar{\nu}_t \right\}, & \text{if } \frac{\partial V}{\partial x}(t, 0) \geq \alpha + \beta N \bar{\nu}_t \\
0, & \text{if } \frac{\partial V}{\partial x}(t, 0) < \alpha + \beta N \bar{\nu}_t.
\end{cases}
\]
The previous quantity \(x^\star(t)\) defines a threshold capacity: producers with installed capacity \(x < x^\star(t)\) at time~\(t\) will invest to increase their capacity, whereas producers with capacity already exceeding this threshold will not invest. In addition, since $\nu^{\star}(T,x)=0$, we deduce $x^\star(T)=0$, and we can define $T^\star=\inf \{t \in [0,T] : x^\star(s)=0, \; \forall s \in [t,T] \}$,
which is the analogue of $T^\star$ in \Cref{T^* lemma} for the homogeneous case. 
Note that at $T^\star$, since $x^\star(T^\star) = 0$, no producer should install new capacity, \textit{i.e.} $\bar{\nu}_{T^\star}=0$. Using this in the definition of $x^\star$ above, we deduce $\frac{\partial V}{\partial x} (T^{\star},0)=\alpha$. Since $\frac{\partial V}{\partial x} (T,0)=0$, we conclude that $T^\star < T$. Note that after $T^\star$, $x^\star$ will stay at $0$, meaning that no producer will install new capacity for the remaining time. This is illustrated in \Cref{x_star_fig} and confirmed by the numerical results in \Cref{linear_section}.
 
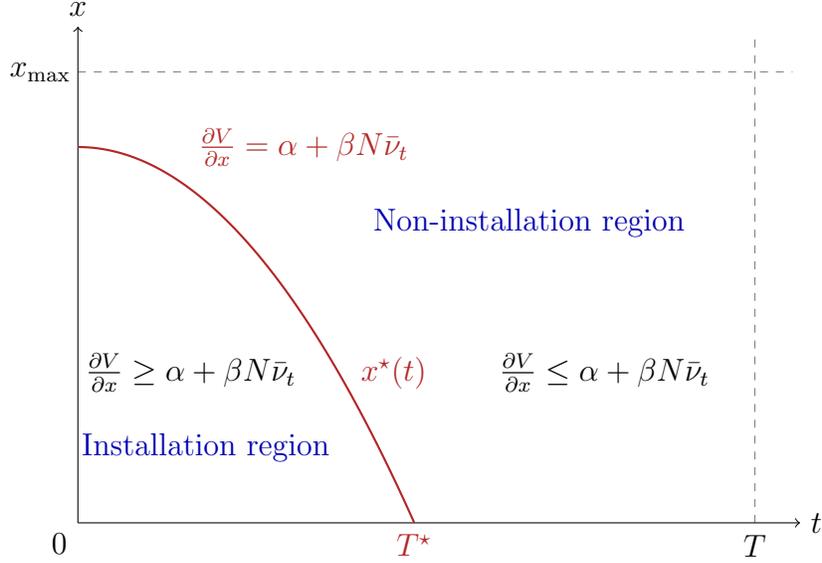
\begin{figure}[ht]
\centering
\begin{tikzpicture}

  \def\xmax{9.5}
  \def\ymax{6.5}
   
  \draw[dashed, gray] (9,0) -- (9,\ymax);

  \draw[dashed, gray] (0,6) -- (\xmax,6);
 
  \draw[->] (0,0) -- (\xmax+0.1,0) node[right] {$t$};
  \draw[->] (0,0) -- (0,\ymax+0.1) node[above] {$x$};

  \node at (9,-0.3) {$T$};
  \node at (-0.5,6) {$x_{\max}$};

  \draw[thick, red, domain=0:4.472135955, samples=100] 
    plot (\x, {max(5 - 0.25*(\x)^2, 0)}) node[below] {$T^\star$};

  \node[align=center] at (7, 2) {$\frac{\partial V}{\partial x}\leq\alpha+\beta N\bar{\nu}_t$};
  \node[align=center] at (1.5, 2) {$\frac{\partial V}{\partial x}\geq\alpha+\beta N\bar{\nu}_t$};
  \node[align=center,red] at (3, 5) {$\frac{\partial V}{\partial x}=\alpha+\beta N\bar{\nu}_t$};
  \node[align=center,red] at (4.2,2){$x^\star(t)$};
  
  \node[blue] at (1.7, 1) {Installation region};
  \node[blue] at (6,4) {Non-installation region};
  \node at (0,0)[below left] {$0$};
\end{tikzpicture}
\caption{Installation and non-installation regions}
\label{x_star_fig}
\end{figure}

In the following proposition, we provide a semi-explicit formula for the value function in the non-installation region, as illustrated in \Cref{x_star_fig}. In particular, this result shows that if a producer does not install new capacity at some time \(t \in [0,T]\), they will not invest at any later time \(s \in [t,T]\). In other words, once a producer's capacity belongs to the non-installation region, it will remain in that region.
\begin{proposition}\label{proposition_V}
In the non-installation region $\{(t,x):x\geq x^\star(t)\}$, $V(t,x)$ is given by 
\begin{equation}\label{integral_V}
    V(t,x)= x\int_{t}^{T} \erm^{-(r+\delta)(s-t)}h(P(N\bar{x}_s+x \erm^{-\delta(s-t)})-c)\, \drm s.
\end{equation}
\end{proposition}
\begin{proof}
First, in the non-installation region, the HJB equation \eqref{HJB} simplifies to
\begin{align*}
\frac{\partial V}{\partial t}-rV-\delta x\frac{\partial V}{\partial x}+h(P(x+N\bar{x}_t)-c)x=0,
\end{align*}
with terminal condition $V(T,x)=0$, which is a linear PDE for which existence and uniqueness holds. It is then straightforward to see that the candidate solution in \eqref{integral_V} indeed satisfies the above PDE. 
Clearly, $V$ is also the value obtained under the trajectory $x_s=x_t \erm^{-\delta(s-t)}$, which coincides with the strategy of no installation between $t$ and $T$. 
\end{proof}

\subsection{Linear price function}\label{linear_section}
In this section, we revisit the case of a linear price function $P(x)=d_1-d_2x$ with $d_1,d_2>0$, previously studied in \Cref{limhomogeneous} for the homogeneous case.
\begin{corollary}\label{cor:V_linear}
In the non-installation region $\{(t,x):x\geq x^\star(t)\}$, the value function is given by $V(t,x)=a_tx^2+b_tx$ with 
\begin{align}\label{a_eq}
    a_t &= -\frac{h d_2}{r + 2\delta} \left( 1 -  \erm^{-(r + 2\delta)(T - t)} \right), \quad 
b_t = h \int_{t}^{T} \erm^{-(r+\delta)(s-t)} (d_1-c-d_2N\bar{x}_s)\, \drm s
\end{align}
and the threshold curve is
\begin{equation}\label{x_star}
x^\star(t)=\max \left( \frac{\alpha+\beta N \bar{\nu}_t-b_t}{2a_t},0\right),\quad\text{ for }t<T.
\end{equation}
\end{corollary}
\begin{proof}
Under the linear price specification, we can make the linear quadratic ansatz $V(t,x)=a_tx^2+b_tx$. From \Cref{proposition_V}, we directly deduce that the pair $(a,b)$ is solution to the following system of ODE,
\begin{equation*}
\dot{a}_t-ra_t-2\delta a_t-hd_2 =0, \quad 
\dot{b}_t-rb_t-\delta b_t+h(d_1-c-d_2\bar{x}_tN)=0, \quad a_T=b_T=0,
\end{equation*} 
and thus given by \eqref{a_eq}. Finally, Equation \eqref{x_star} is derived from the fact that $V$ satisfies $\frac{\partial V}{\partial x} {(t,x^\star{(t)})}=\alpha+\beta N\bar{\nu}_t$ when $x^\star(t)>0$.
\end{proof}

From the previous result, note that $a_t$ is only a function of $t$, while $b_t$ is expressed as a function of the trajectory of the mean installed capacity, namely $\bar{x}_s$ for $s \in [t,T]$. Unfortunately, such explicit derivation is not possible in the installation region. 
More precisely, on $\{(t,x):x< x^\star(t)\}$, we have by definition that $\nu^\star>0$, and the HJB equation \eqref{HJB} satisfied by the value function $V$ becomes
\begin{equation}\label{new_hjb}
\frac{\partial V}{\partial t}-rV-\delta x\frac{\partial V}{\partial x}+h(d_1-c-d_2(x+\bar{x}_tN))x+\frac{1}{4\beta} \bigg(\dfrac{\partial V }{\partial x}-\alpha-\beta N\bar{\nu}_t \bigg)^2=0.
\end{equation}
However, to our knowledge, no explicit solution exists for the previous HJB equation when the boundary condition is determined by the solution in the non-installation region, given by \Cref{cor:V_linear}. Nevertheless, in the following, we suggest two numerical approaches to approximate the value function, the corresponding optimal strategy, and the solution to \eqref{FP} in that region. The first approach uses a quadratic ansatz for approximation, while the second employs a finite difference scheme.

\subsubsection{Linear-quadratic approximation }\label{two_piece}

For the first numerical approach, we approximate the solution to PDE \eqref{new_hjb} by a quadratic function. More precisely, using the ansatz $V(t,x)= A_tx^2+B_tx+C_t$ in \eqref{new_hjb}, we obtain that the functions $A,B,C$ should be solution to the system 
\begin{subequations}\label{ODE_installation}
\begin{align}
\dot{A}_t &= (r+2\delta)A_t + hd_2 - \frac{1}{\beta} A_t^2 \label{Aequation} \\
\dot{B}_t &= rB_t + \delta B_t - h(d_1-c-d_2\bar{x}_tN) - \frac{1}{\beta} A_t(B_t-\alpha-\beta N\bar{\nu}_t) \label{Bequation}\\
\dot{C}_t &= rC_t - \frac{1}{4\beta} \big( (B_t-\alpha-\beta N\bar\nu_t)^+ \big)^2.\label{Cequation}
\end{align} 
\end{subequations}
By \eqref{Aequation}, \(A\) solves a Riccati equation, and therefore it will be of the form
\[
{A_t = \beta  \frac{\lambda_1 + R_A \lambda_2  \erm^{(\lambda_2 - \lambda_1)t}}{1 + R_A  \erm^{(\lambda_2 - \lambda_1)t}}}
\]
with $\lambda_1>0>\lambda_2$ solutions to the quadratic equation $\lambda^2- \beta (r+2\delta)\lambda-h \beta d_2=0$ and $R_A$ a constant to be determined. Similarly, from \eqref{Bequation}, $B$ depends on a constant $R_B$, which is the value of $B$ at $T^\star$ and is similarly not known yet. We choose the above constants $R_A$ and $R_B$ so that \(A_{T^\star} = 0\), since concavity of \(V\) requires \(A \leq 0\), and \(B_{T^\star} = b_{T^\star}\), to ensure that \(\frac{\partial V}{\partial x}\) matches at the point \(\bigl(T^\star,0\bigr)\) with the solution in the non-installation region. Note that we do not impose the smooth pasting conditions along the entire threshold boundary \(x^\star(t)\), but only at \(T^\star\); this is precisely what makes the solution an approximation. In fact, if we attempted to match both $V$ and $\frac{\partial V}{\partial x}$ continuously along the entire threshold $x^*(t)$, we would arrive at a contradiction, which is the reason why this is not an exact solution but an approximation.
Since the optimal control depends only on \(\frac{\partial V}{\partial x}\), we do not need to need to compute the solution to \eqref{Cequation} for approximating the mean field equilibrium.

Using this quadratic approximation for \(V\), we first solve the Fokker–Planck equation~\eqref{FP} by finite differences, as described below, approximating \(\frac{\partial V}{\partial x}\) by \(2A_t x + B_t\) in the installation region. Next, we update the mean installation rate via the fixed-point equation~\eqref{nubarformula}, which becomes under the quadratic approximation
\[
\bar{\nu}_t = \frac{ \int_{0}^{x^\star(t)} \bigl(2A_t x + B_t - \alpha\bigr)\,m(t,x)\,\mathrm{d}x}{2\beta + \beta N \int_{0}^{x^\star(t)} m(t,x)\,\mathrm{d}x}, \quad t \in [0,T],
\]
and use this updated value to iterate the algorithm until convergence is achieved.

\paragraph{Summary of the algorithm.}\label{algorithm}

\begin{enumerate}\itemsep -2pt
    \item \textit{Initialization.} Start with an initial guess for $\bar{\nu}$, for example $\bar{\nu}_t=\bar{x}_0(1-t/T)$.
    \item \textit{Mean state dynamics.} Solve numerically $\dot{\bar{x}}_t = - \delta \bar{x}_t+\bar{\nu}_t$ starting from $\bar{x}_0 = x_0$ to obtain the average capacity $\bar{x}_t$.
    \item \textit{Non-installation region.} For all $t$, compute $b_t$ numerically from \eqref{a_eq}, and use it together with $a_t$ in \eqref{x_star} to determine the threshold $x^\star(t)$. Deduce then $T^\star$ by solving $x^\star(t)=0$.
    \item \textit{Installation region.} Use the quadratic ansatz introduced above where $A$, $B$, and $C$ are computed by backward integration from terminal conditions $A_{T^\star} = 0$, $B_{T^\star} = b_{T^\star}$ of the associated ODEs \eqref{ODE_installation}.
    \item \textit{Fokker–Planck equation.} Starting from the known repartition at time $0$, solve forward in time
    \[
    \frac{\partial m}{\partial t} + \frac{\partial}{\partial x} \left[ m\left( -\delta x + \frac{(2A_t x + B_t - \alpha - \beta N \bar{\nu}_t)^+}{2\beta} \right) \right] = 0,
    \]
    which corresponds to the FP equation in \eqref{FP}, to obtain the density $m$. Here we use an explicit Euler finite differences scheme with upwinding for the $x$ derivative. 
    \item \textit{Update and iterate.} Use the density computed at the previous step in \eqref{nubarformula} to update $\bar{\nu}$. Return to step 2 with the updated $\bar{\nu}$ and iterate until convergence.
\end{enumerate}

We present a numerical example in which we use the same parameter values as in \Cref{numerics}, see \eqref{eq:param_numerics}, together with $d_1=500$\$/MWh, $d_2=10^{-2}$\$/MW$^2$h specifying the linear price function. For the initial distribution of installed capacities, we consider a truncated exponential distribution. More precisely, we define \(N\) 
discrete capacity levels \(x_i = i \cdot \Delta x\) for \(i = 1,\dots,N\), where \(\Delta x = x_{\text{end}}/N\) and \(x_{\text{end}}\) denotes a maximum capacity.  We assign exponential weights of the form $w_i = \erm^{-N x_i / X_0}$, which reflect the fact that smaller producers are more common in renewable electricity markets. We then normalize and rescale these weights to define each $x_0^{(i)}$ so that the total initial capacity equals \(X_0 = 30\,\text{GW}\).
This yields a discrete approximation of a truncated exponential distribution over the grid \(\{x_i\}\), with mass concentrated near small capacities and a long tail toward larger producers.

\begin{figure}[h!]
    \centering
    
    \begin{subfigure}[b]{0.30\textwidth}
        \centering
        \includegraphics[width=\textwidth]{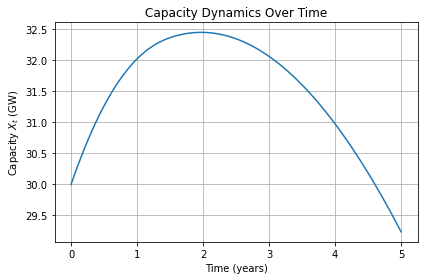}
        \caption{\footnotesize $T=5$ years}
        \label{fig:verification1}
    \end{subfigure}
    \begin{subfigure}[b]{0.30\textwidth}
        \centering
        \includegraphics[width=\textwidth]{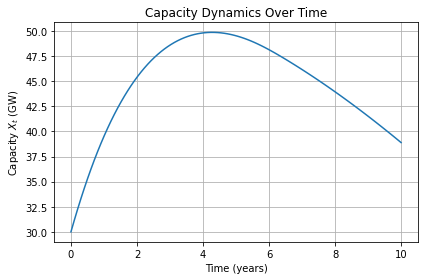}
        \caption{$T=10$ years}
        \label{fig:verification2}
    \end{subfigure}
    \hfill
    \caption{Capacity dynamics with linear price function for different time horizons}
\end{figure}

The resulting capacity trajectories are presented in \Cref{fig:verification1,fig:verification2}, for horizons \(T=5\) and \(T=10\) years, respectively. Based on the numerical computations, we identify the times at which all producers cease installing new capacity: \(T^\star \approx 1.71\) in the first case ($T=5$) and \(T^\star \approx 6.26\) in the second ($T=10$). Beyond these thresholds, no capacity is installed, and the total capacity declines at the natural rate \(\delta\). 

\begin{figure}[ht]
    \centering
    
    \begin{subfigure}[b]{0.32\textwidth}
        \centering
        \includegraphics[width=0.99\textwidth]{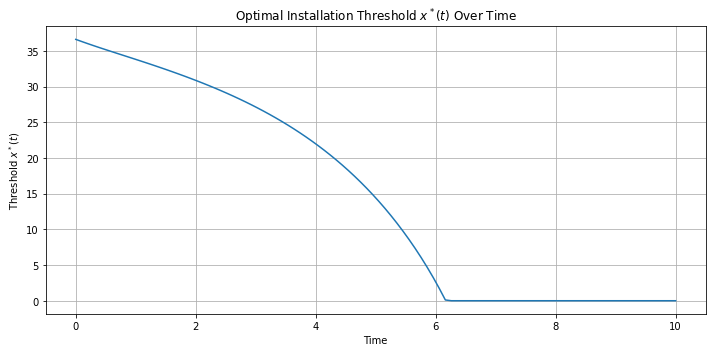}
        \caption{\footnotesize Threshold \(x^\star(t)\)}
        \label{fig:verification30}
    \end{subfigure}
    \hfill
    \begin{subfigure}[b]{0.32\textwidth}
        \centering
        \includegraphics[width=\textwidth]{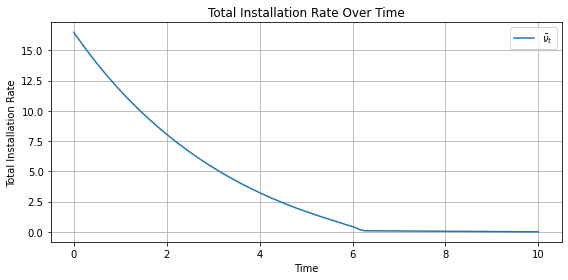}
        \caption{\footnotesize Total installation rate ${K}_t$}
        \label{fig:verification3}
    \end{subfigure}
    \hfill
    \begin{subfigure}[b]{0.32\textwidth}
        \centering
        \includegraphics[width=\textwidth]{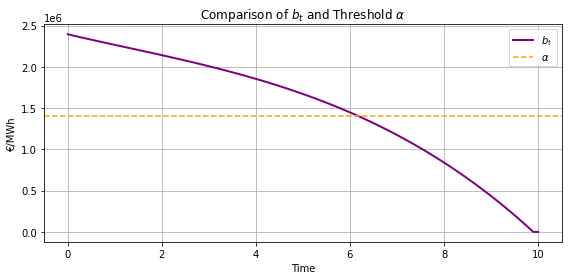}
        \caption{\footnotesize Comparison of $b_t$ and $\alpha$}
        \label{fig:verification4}
    \end{subfigure}
    \caption{Threshold between regions and installation rate for $T=10$ years}
\end{figure}

For the second case $T=10$, we also graphically represent in \Cref{fig:verification30} the threshold \(x^\star(t)\) between the installation and non-installation regions, and observe that its behavior is similar to the illustration in \Cref{x_star_fig}. As shown in \Cref{fig:verification3}, the mean installation rate drops to zero for \(t \geq T^\star\), confirming that no capacity is installed beyond this time. Finally, still considering \(T=10\), we also plot in \Cref{fig:verification4} the function $b$ given by \eqref{a_eq} in \Cref{cor:V_linear}. We observe that \(b_{T^\star} \approx \alpha\), as expected. Indeed, this follows from the fact that \(x^\star(T^\star)=0\), $\bar \nu_{T^\star} = 0$ and the characterization of the threshold curve in \eqref{x_star}.

\subsubsection{Finite differences method}\label{central}
The other method relies on a classical finite difference scheme. Specifically, we discretize the HJB equation~\eqref{HJB} backward in time on a uniform grid over both time and capacity. At each time step, the spatial derivatives \(\frac{\partial V}{\partial x}\) are approximated using central differences, which leads to a nonlinear system of equations for the unknown values \(V_n\) at the current time layer. This system is then solved using a root-finding algorithm (\textit{e.g.}, \texttt{fsolve}).

More precisely, given a uniform grid in time \(t_n\) and space \(x_j\) with step sizes \(\Delta t\) and \(\Delta x\), we approximate:
\[
\frac{\partial V}{\partial t} (t_n, x_j) \approx \frac{V^{n+1}_j - V^{n}_j}{\Delta t},
\quad
\frac{\partial V}{\partial x} (t_n, x_j) \approx \frac{V^n_{j+1} - V^n_{j-1}}{2\Delta x}.
\]
Accordingly, the discretized HJB equation at each grid point \((t_n, x_j)\) takes the form:
\begin{align*}
\frac{V^{n+1}_j - V^{n}_j}{\Delta t}
- r V^n_j
- \delta x_j \frac{V^n_{j+1} - V^n_{j-1}}{2\Delta x}
&+ (P(N \bar{x}_n + x_j) - c) hx_j\\
&+ \frac{1}{4\beta} \left[\left( \frac{V^n_{j+1} - V^n_{j-1}}{2\Delta x} - \alpha - \beta N \bar{\nu}_n \right)^+\right]^2
= 0.
\end{align*}
At each time step, the system is solved backward in time, using \(V_{n+1}\) to compute \(V_n\). We further impose the following boundary conditions at the boundaries of the domain:
\begin{itemize}\itemsep -2pt
\item[\scriptsize$\bullet$] at $x=0$, we match the value to the solution $C_{t_n}$ obtained by \eqref{Cequation};
\item[\scriptsize$\bullet$] at $x=0$, we match the value to the solution $C_{t_n}$ obtained by \eqref{Cequation};
\item[\scriptsize$\bullet$] at $x=x_{\rm max}$, we impose $V_n(x_{\max})=a_{t_n}x_{\max}^2+b_{t_n}x_{\max}$;
\item[\scriptsize$\bullet$] at $T=0$, we impose $V_{T}^{j}=0$.
\end{itemize}

The results are illustrated in \Cref{fig:summary} for the following parameters:
\begin{equation}
\begin{split}\label{model_param_2}
x_0 &=0.1\text{MW},\quad c=1\$/\text{MWh}, \quad r=0.05\text{year}^{-1},\quad
\delta=\frac{\log2}{10}\text{year}^{-1}, \quad N=10, \\
h&=1\frac{\text{hours}}{\text{year}},\quad \alpha=0.1\$/\text{MW},\quad\frac{1}{2\beta}=5\text{MW}^2/(\$\text{year}),\quad T=1 \text{ years.}
\end{split}
\end{equation}
The linear price function is further characterized by $d_1=2\$/\text{MWh}$ and $d_2= 1\$/\text{MW}h$. 

\Cref{mean_1,mean_2} display the trajectories of the mean capacity and the optimal control over time. For comparison, we also solve the system numerically using the alternative method that approximates \(V\) by two quadratic functions. As shown in \Cref{mean_3,mean_4}, this approach produces results that are very similar to those obtained with the finite difference scheme. This agreement verifies that both numerical methods are accurate and consistent with each other. 

In addition, \Cref{distribution_1,distribution_2} show the distribution of capacities at \(t=0\) and \(t=1\). As intended, the initial distribution is exponential, and over time it becomes more concentrated around higher capacities, since producers with smaller initial capacities tend to invest more rapidly. 

\begin{figure}[ht]
    \centering
    
    \begin{subfigure}[b]{0.32\textwidth}
        \centering
        \includegraphics[width=\textwidth]{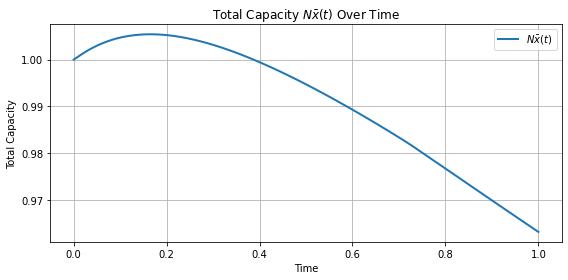}
        \caption{\footnotesize Total capacity $X$}
        \label{mean_1}
    \end{subfigure}
    \hfill
    \begin{subfigure}[b]{0.32\textwidth}
        \centering
        \includegraphics[width=\textwidth]{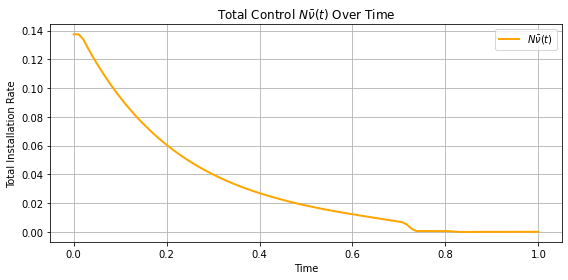}
        \caption{\footnotesize Total installation rate $K$}
        \label{mean_2}
    \end{subfigure}
    \hfill
    \begin{subfigure}[b]{0.32\textwidth}
        \centering
        \includegraphics[width=\textwidth]{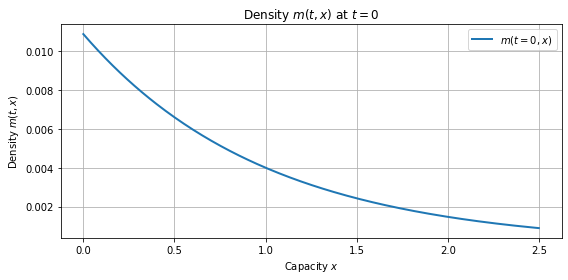}
        \caption{\footnotesize Density $m(t=0,x)$}
        \label{distribution_1}
    \end{subfigure}

    \vspace{1em}

    \begin{subfigure}[b]{0.32\textwidth}
        \centering
        \includegraphics[width=\textwidth]{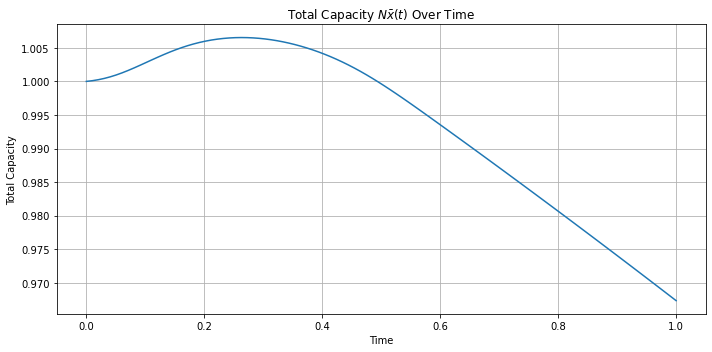}
        \caption{\footnotesize $X$ by quadratic approach}
        \label{mean_3}
    \end{subfigure}
    \hfill
    \begin{subfigure}[b]{0.32\textwidth}
        \centering
        \includegraphics[width=\textwidth]{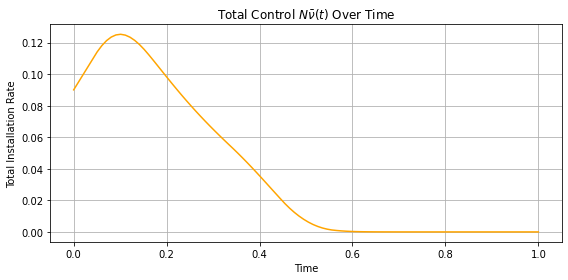}
        \caption{\footnotesize $K$ by quadratic approach}
        \label{mean_4}
    \end{subfigure}
    \hfill
    \begin{subfigure}[b]{0.32\textwidth}
        \centering
        \includegraphics[width=\textwidth]{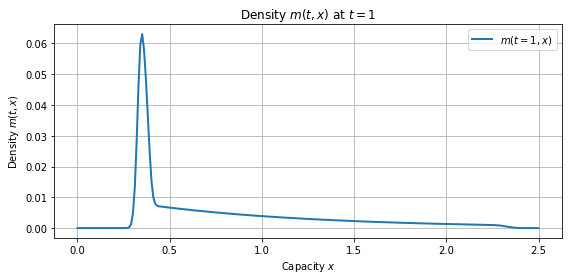}
        \caption{\footnotesize Density $m(t=1,x)$}
        \label{distribution_2}
    \end{subfigure}
    \caption{Simulation results with linear price and parameter values \eqref{model_param_2}}
    \label{fig:summary}
\end{figure}

\subsection{Inverse price function }\label{inverse_price}
We conclude this section by returning to the inverse price case, examined in \cite{alasseur2023mean} and previously in \Cref{numerics} for the homogeneous case. The main theoretical result presented below directly follows from \Cref{proposition_V}, using the specification $P(x) = p/x$ for the price function. 

\begin{corollary}
In the non-installation region $\{(t,x):x\geq x^\star(t)\}$, $V(t,x)$ is given by 
\[
V(t,x) = h x \left( p \int_t^T \frac{ \erm^{-(r+\delta)(s-t)}}{N \bar{x}_s + x  \erm^{-\delta(s - t)}}\,  \drm s - \dfrac{c}{r + \delta} \big( 1 -  \erm^{-(r + \delta)(T - t)} \big) \right).
\]
\end{corollary}

Therefore, in the non-installation region, we can further compute
\begin{align*}
\frac{\partial V}{\partial x}
&=h\left( p \int_t^T \frac{ \erm^{-(r+2\delta)(s-t)}N\bar{x}_s}{(N \bar{x}_s + x  \erm^{-\delta(s - t)})^2}\,  \drm s - \dfrac{c}{r + \delta} \big( 1 -  \erm^{-(r + \delta)(T - t)} \big) \right)
\end{align*}
which is strictly decreasing in $x$, confirming that $V$ is indeed concave. Then, \(x^\star(t)\) should be determined as the solution, if it exists, to $\frac{\partial V}{\partial x} (t, x^\star(t)) = \alpha + \beta N\,\bar{\nu}_t$, $t \in [0,T]$.
However, contrary to the results for the linear price case in \Cref{cor:V_linear}, no explicit analytic expression for the threshold function can be derived here. Nevertheless, $x^\star(t)$ can be approximated numerically. Moreover, since there is no obvious ansatz to approximate the solution to the HJB equation in the installation region, we focus here on the second numerical method, based on finite differences. Numerical results are presented for the same parameter values used in the previous section, as specified in \Cref{model_param_2}, but with \(p = 2\,\$/\text{h}\) to characterize the inverse price function. This value of $p$ is chosen so that the inverse price $p/x$ is of the same order as the linear price $d_1 - d_2 x$ for $x$ near \(X_0\), the initial total capacity, ensuring that the results remain comparable in terms of dynamics and price levels. The trajectories of the mean capacity and the mean installation rate are shown in \Cref{finite_p/x,finite_p/x_1} respectively.

\begin{figure}[ht]
    \centering
    \begin{subfigure}{0.4\textwidth}
        \centering
        \includegraphics[width=\textwidth]{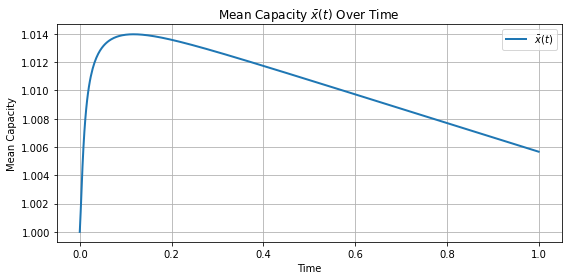}
        \caption{\footnotesize Total capacity $X(t)$}
        \label{finite_p/x}
    \end{subfigure}
  \begin{subfigure}{0.4\textwidth}
      \centering
        \includegraphics[width=\textwidth]{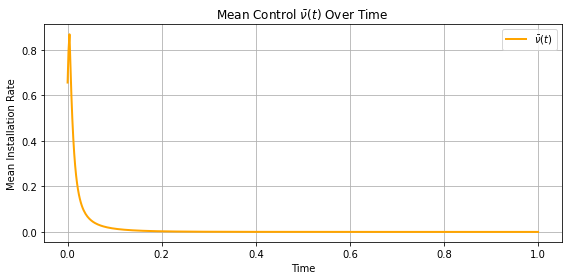}
        \caption{\footnotesize Total installation rate $K_t$}
       \label{finite_p/x_1}
  \end{subfigure}
  \caption{Simulation results with inverse price and parameter values \eqref{model_param_2}}
    \label{fig:summary_inverse}
\end{figure}

The simulation shows that the inverse price function \( P(x) = \frac{p}{x} \) leads to a strong incentive for early installation: when aggregate capacity \( x + N\bar{x}_t \) is low, the electricity price becomes very high, resulting in steep marginal rewards for producers. This creates a pronounced front-loading in installation behavior, as producers seek to take advantage of the initially high prices before crowding drives them down. Consequently, total production ramps up sharply in the early part of the horizon.
\section{Production intermittency}\label{sec:randomness}
In this section, we briefly study a version of the model with randomness. Specifically, suppose that \(x_t\) evolves according to
\[
\mathrm{d}x_t = \big(-\delta x_t + \nu_t\big) \mathrm{d}t + \sigma x_t \mathrm{d}W_t, \quad t \in [0,T],
\]
where $W$ is a Brownian Motion.
Introducing randomness in the installed capacity \(x_t\) indirectly generates volatility in the instantaneous profit \(h x_t \bigl(P(X_t)-c\bigr)\). In this sense, incorporating stochastic fluctuations in the state variable \(x_t\) could serve to model the intermittency of renewable energy sources—such as variations in solar or wind output—which leads to uncertain revenue streams. 
The variability from intermittency of renewables (see, for instance \citeayn{carmonayang}) increases the sensitivity of grids to fluctuations. Large-scale expansions, without corresponding flexibility measures, can make systems more exposed to operational risks (see \citeayn{orfeus}). 

The value function should now be defined as follows,
\begin{align*}
V(t,x) = \sup_{\nu \geq 0} \mathbb{E} \left[ \left. \int_t^T  \erm^{-r(s-t)} \left( h(P(x_s + N\bar{x}_s) - c)x_s - \nu_s(\alpha + \beta(\nu_s + N\bar{\nu}_s)) \right)  \drm s \,\right|\, x_t = x \right],
\end{align*}
for all $(t,x) \in [0,T] \times \R$. The form of the optimal control remains the same as \eqref{nustar}, and the corresponding HJB equation is similar to \eqref{HJB},
\begin{align*}
\frac{\partial V}{\partial t}-rV-\delta x\frac{\partial V}{\partial x}+hx(P(x+\bar{x}_tN)-c)+\frac{[(\frac{\partial V }{\partial x}-\alpha-\beta N\bar{\nu}_t)^+]^2}{4\beta}+\frac{1}{2}\sigma ^2x^2\frac{\partial ^2V}{\partial x^2}=0,
\end{align*}
with terminal condition $V(T,x) = 0$. More precisely, if we let $\sigma = 0$ in the above HJB equation, we recover \eqref{HJB}. Similarly, the Fokker-Planck becomes
\begin{align*}
\frac{\partial }{\partial t}m+\frac{\partial }{\partial x}\left[m\left(-\delta x+\frac{(\frac{\partial V}{\partial x}-\alpha-\beta N\bar{\nu}_t)^+}{2\beta}\right)\right]=\frac{1}{2}\frac{\partial ^2}{\partial x^2}\left(\sigma^2x^2m\right).
\end{align*}
Finally, the fixed point condition coincides with \eqref{nubarformula}. Intuitively, this model is expected to yield outcomes that closely resemble those of the deterministic setting studied in the previous sections. For conciseness, we limit ourselves in this section to presenting a numerical example that illustrates this resemblance. Specifically, when the price is linear and \(\sigma^2 < r + 2\delta\), the two numerical methods described earlier can be implemented.

The analogue of the first method discussed in \Cref{two_piece} is to approximate $V$ by a quadratic function in each of the two regions, namely 
\[
V(t,x) =
\begin{cases}
A_tx^2+B_tx+C_t,& \text{if } \frac{\partial V}{\partial x}(t, x) \geq \alpha + \beta N \bar{\nu}_t \\
a_tx^2+b_tx, & \text{if } \frac{\partial V}{\partial x}(t, x) < \alpha + \beta N \bar{\nu}_t,
\end{cases}
\]
with $b,B,C$ solutions to the same equations (\ref{a_eq}, \ref{Bequation}, \ref{Cequation}) as before, but $a,A$ now solve 
\begin{align*}
\dot{a}_t&=(r+2\delta-\sigma^2)a_t+hd_2, \quad 
\dot{A}_t=(r+2\delta-\sigma^2)A_t+hd_2-\frac{A_t^2}{\beta}.
\end{align*}
Then, we can apply the numerical method of \Cref{algorithm} and solve these two equations.

Alternatively, we can directly discretize the HJB equation using a central finite difference scheme, following the approach described in \Cref{central}. The resulting discretized HJB equation takes the form
\begin{align*}
&\ \frac{V^{n+1}_j - V^{n}_j}{\Delta t}
- r V^n_j
- \delta x_j \frac{V^n_{j+1} - V^n_{j-1}}{2\Delta x}
+  (P(N \bar{x}_n + x_j) - c) hx_j \\
+ &\ \frac{1}{4\beta} \left[\left( \frac{V^n_{j+1} - V^n_{j-1}}{2\Delta x} - \alpha - \beta N \bar{\nu}_n \right)^+\right]^2
+ \frac{\sigma^2 x_j^2}{2} \frac{V^n_{j+1} - 2V^n_j + V^n_{j-1}}{(\Delta x)^2}
= 0.
\end{align*}

For illustration, we apply the first method with the same parameters as in \Cref{two_piece}, with \(T=10\) and $\sigma =0.4$. The resulting capacity is shown in \Cref{mean_stochastic_}, and the threshold \(x^\star(t)\) is presented in \Cref{star_stochastic_}. We observe that the capacity trajectory with randomness (orange dotted curve) closely resembles the deterministic trajectory (blue curve). However, \(x^\star(t)\) is noticeably smaller compared to \Cref{fig:verification30}. This difference arises because \(a_t\) becomes much larger when \(\sigma>0\), since \((r + 2\delta - \sigma^2) < (r + 2\delta)\).

\begin{figure}[ht]
\centering
    \begin{subfigure}[b]{0.45\textwidth}
        \centering
        \includegraphics[width=\textwidth]{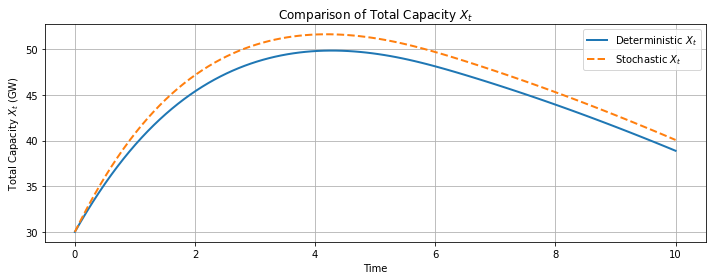}
        \caption{\footnotesize Total capacity ${X}_t$}
        \label{mean_stochastic_}
    \end{subfigure}
    \begin{subfigure}{0.45\textwidth}
        \centering
        \includegraphics[width=\textwidth]{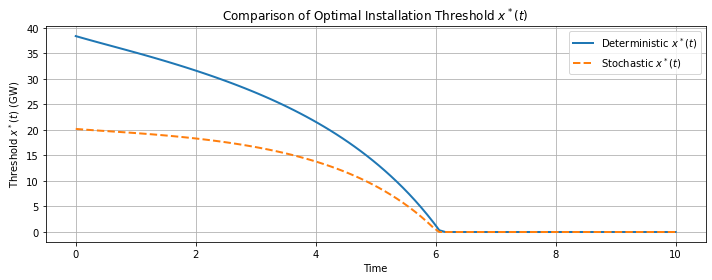}
        \caption{\footnotesize Threshold $x^\star(t)$}
        \label{star_stochastic_}
    \end{subfigure}
    \caption{Capacity dynamics and threshold for linear price function and $T=10$}
\end{figure}

This reduction reflects agents’ increased caution under uncertainty. In particular, the stochastic component in the dynamics introduces a precautionary effect, where firms may delay or reduce investments to hedge against downside risks. This behavior is consistent with real-world observations, where firms facing high uncertainty (\textit{e.g.}, due to variable output from renewables or policy instability) often exhibit more conservative expansion strategies.

\section{Conclusion}\label{sec:conclusion}

We conclude with a few modeling comments that suggest possible extensions or refinements of the current setup.

While we model the electricity price as a function of total installed capacity only, a more general formulation could allow the price to depend explicitly on time, \textit{i.e.} $ P = P(t, X_t)$, to capture external seasonal effects, regulatory interventions, or trends in demand that are not driven solely by supply. Incorporating such time dependence could be particularly relevant for settings where electricity prices exhibit strong periodic or structural trends.

Another possible extension is to introduce heterogeneity across firms in terms of cost structures, risk preferences, or technologies. For instance, the marginal installation cost \( \alpha \) could be modeled as a decreasing function of installed capacity \( x \), reflecting economies of scale. Similarly, the crowding sensitivity parameter \( \beta \) could also be specified as \( \beta(x) \), decreasing in \( x \), to capture the idea that larger firms may face lower relative frictions due to better access to contractors, permits, or supply chains. This would allow larger producers to install capacity more efficiently than smaller ones, capturing realistic differences in incentives between incumbents and new entrants. 

Taking into account storage investment decisions and flexible resources (\textit{e.g.} demand response, dispatchable backup) into the model would allow to assess their interaction with intermittent renewable deployment. Recent events—such as the major power outage in Spain and Portugal triggered by grid imbalances and a lack of short-term flexibility \cite{bajo-buenestado2025iberian}—highlight the critical role of these mechanisms in ensuring system stability during high-renewable penetration.

Finally, it could be relevant to study the impact of various policy tools—such as carbon taxes, renewable subsidies, or capacity remuneration mechanisms—on long-term investment decisions within the mean field framework. For instance, as in \cite{alasseur2023mean}, one could model subsidies that reduce the marginal cost of production \( c \) or the installation cost \( \alpha \), and examine how these interventions shift the equilibrium and accelerate renewable deployment.

\bibliography{bibliography_Emma} 
\end{document}